\documentclass[10pt,a4paper,english]{article}
\usepackage{amsmath,amscd,amsfonts,eucal,latexsym,amssymb,mathrsfs,mathabx}
\usepackage{amsxtra, amsthm, calc, bbm}
\usepackage[latin1]{inputenc}
\usepackage[active]{srcltx}
\numberwithin{equation}{section}
\newlength{\papersidemargin}
\newlength{\papertopmargin}
\newlength{\maxcarwidth}
\theoremstyle{definition}

\theoremstyle{plain}

\theoremstyle{definition}
\newtheorem{Remark}{Remark}[section]

\theoremstyle{remark}
\newtheorem{example}{Example}[section]
\theoremstyle{definition}
\newtheorem{definition}{Definition}[section]
\theoremstyle{plain}
\newtheorem{theorem}{Theorem}[section]
\newtheorem{proposition}{Proposition}[section]

\newtheorem{corollary}{Corollary}[section]

\newcounter{remcount}

\newcounter{propcount}

\newlength{\maxlabelwidth}

\def\cI{{\cal I}}

\def\cR{{\cal R}}

\def\cU{{\cal U}}
\def\cV{{\cal V}}
\def\cW{{\cal W}}

\def\cZ{{\cal Z}}
\def\bE{{\mathbb E}}
\def\bF{{\mathbb F}}

\def\bH{{\mathbb H}}

\def\bL{{\mathbb L}}
\def\bN{{\mathbb N}}

\def\bQ{{\mathbb Q}}
\def\bR{{\mathbb R}}

\def\bV{{\mathbb V}}
\def\bW{{\mathbb W}}
\def\bZ{{\mathbb Z}}
\def\a{\alpha}

\def\g{\gamma}        
        \def\D{\Delta}
 
    \def\ep{\epsilon}

\def\l{\lambda}       
\def\m{\mu}

\def\rr{\rho}
\def\s{\sigma}

\def\o{\omega}        \def\O{\Omega}
\def\fB{{\mathfrak B}}
\def\fC{{\mathfrak C}}

\newcommand{\sA}{\mathscr{A}}

\newcommand{\inst}[1]{$^\textrm{#1}$ }
\newcommand{\email}[1]{e-mail: #1}
\newcommand{\fffC}{\fC_c(\bE)}
\newcommand{\ffC}{\fC(\bE)}
\newcommand{\co}{\text{co}}
\newcommand{\cco}{\overline{\text{co}}}
\newcommand{\aint}{\text{(A)}\negthickspace\int_{\O}\negthickspace\negthickspace}
\newcommand{\hint}{\text{(H)}\negthickspace\int_{\O}\negthickspace\negthickspace}
\newcommand{\Aint}{(\text{A})\negthickspace\int_{\O^{\prime}}\negthickspace\negthickspace}
\newcommand{\Hint}{(\text{H})\negthickspace\int_{\O^{\prime}}\negthickspace\negthickspace}
\newcommand{\mi}{\overset{.}{+}}
\begin{document}
\title{Convexification in the limit and strong law\\
 of large numbers for closed-valued\\
random sets in Banach spaces}
\author{Francesco S. de Blasi\inst{1}
\and Luca Tomassini\inst{2}}

\date{
\parbox[t]{0.9\textwidth}{\footnotesize{%
\begin{enumerate}
\renewcommand{\theenumi}{\arabic{enumi}}
\renewcommand{\labelenumi}{\theenumi}
\item[1.] Dipartimento di Matematica, Universit\`a di Roma ``Tor Vergata'', Via della Ricerca Scientifica, 1, I-00133 Roma, Italy, \email{deblasi@mat.uniroma2.it}
\item[2.] Dipartimento di Scienze, Universit\`a di
Chieti-Pescara G. d'Annunzio, I-65127, \email{tomassini@sci.unich.it} 
\end{enumerate}
}}
\\
\vspace{\baselineskip}
%\today
}

\maketitle

\section{Introduction}\label{section1}

The study of the strong law of large numbers (SLLN) for compact-valued random sets was started by Artstein and Vitale \cite{Artstein1975}, with their seminal 1975 paper. Ever since, significant extensions and developments have been obtained by several authors, including Puri and Ralescu \cite{Puri1983, Puri1985}, Cressie \cite{Cressie1978}, Hess \cite{Hess1979, Hess1999}, Artstein and Hansen \cite{Artstein1985}, Hiai \cite{Hiai1984}, Ter\'an and Molchanov \cite{Teran2006}. A systematic presentation of the status of the theory and the multivalued analysis prerequisites can be found in the monographs of Molchanov \cite{Molchanov2005}, Hu and Papageorgiou \cite{Hu1997} and Castaing and Valadier \cite{Castaing1977}.

In the SLLN for compact-valued random sets, it is somehow natural to investigate convergence in the sense of the Pompeiu-Hausdorff metric $h$. This type of convergence does not seem to be appropriate for random sets which are merely bounded and closed-valued and thus, in this case, the SLLN has been studied by using weaker modes of convergence, like Wijsman's convergence \cite{Wijsman1966} or Mosco's convergence \cite{Mosco1969, Beer1993}.

In the present paper we consider a SLLN for random sets with bounded and closed values contained in an arbitrary (not necessarily separable) Banach space and, in this context, we shall use a notion of convergence, namely Fisher's convergence, which is stronger than Wijsman's convergence but in general not comparable with Mosco's convergence.

More precisely, denote by $\bE$ an arbitrary real Banach space and by $\fC(\bE)$ (resp. $\fC_{c}(\bE)$) the space of all nonempty bounded closed (resp. nonempty bounded closed and convex) subsets of $\bE$ endowed with the Pompeiu-Hausdorff metric $h$. Let $(\O, \sA, P)$ be a complete probability space without atoms and let $\cU$ be a closed and separable subset of $\fC(\bE)$. Let $\{X_n\}_{n=1}^{+\infty}$ be a sequence of independent identically distributed (i.i.d.) random sets $X_n\; :\; \O\rightarrow\cU$. If, in addition, for $\o\in\O$ almost surely (a.s.) the sequence $\{X_n(\o)\}_{n=1}^{+\infty}$ is compact in $\fC(\bE)$, then it is shown (Theorem \ref{theo6.2}) that
\begin{equation}\label{eq1.1}
 \frac{\overline{X_1(\o)+\dots +X_n(\o)}}{n}\xrightarrow{F} C \qquad\qquad\qquad  \text{a.s. on}\;\O.
\end{equation}
Here the convergence is in the sense of Fisher, and $C\in\fC_{c}(\bE)$ is the expectation of $X_1$. Moreover, if $C$ is compact, the convergence in \eqref{eq1.1} is in the Pompeiu-Hausdorff metric $h$.

To prove the above result, we first consider the corresponding sequence of random sets $\{\overline{co}X_n\}_{n=1}^{+\infty}$ and, using the R\text{\r{a}}dstr$\Ddot{\text o}$m's embedding theorem, we show (Theorem \ref{theo6.1}) that
\begin{equation*}
 \frac{\overline{\overline{co}X_1(\o)+\dots +\overline{co}X_n(\o)}}{n}\xrightarrow{h} C \qquad\qquad\qquad \text{a.s. on}\;\O,
\end{equation*}
where the convergence is in the Pompeiu-Hausdorff metric $h$. Hence, a fortiori,
\begin{equation}\label{eq1.2}
 \frac{\overline{\overline{co}X_1(\o)+\dots +\overline{co}X_n(\o)}}{n}\xrightarrow{F} C \qquad\qquad\qquad \text{a.s. on}\;\O,
\end{equation}
From this, by virtue of a convexification in the limit result (Theorem \ref{theo4.1}), it follows that
\begin{equation}\label{eq1.3}
 \frac{\overline{X_1(\o)+\dots +X_n(\o)}}{n}\xrightarrow{F} C \qquad\qquad\qquad \text{a.s. on}\;\O,
\end{equation}
and thus \eqref{eq1.1} is valid.

We observe that the precompactness of $\{X_n(\o)\}_{n=1}^{+\infty}$ in $\fC(\bE)$ is actually required in order to deduce \eqref{eq1.3} from \eqref{eq1.2}. Furthermore, it is worth noting that the separability of $\cU$ in $\fC(\bE)$ does not imply that the set $U=\cup\{X\; |\; X\in\cU\}$ be separable in $\bE$ and, similarly, the compactness of the sequence $\{X_n(\o)\}_{n=1}^{+\infty}$ does not imply that the set $V=\cup\{X_n(\o)\;|\; n\in\bN\}$ be precompact in $\bE$. To see this it suffices to consider $\cU=\{\overline{B}\}$ and $X_n(\o)=\overline{B}$ for each $n\in\bN$, where $\overline{B}$ is the closed unit ball of an infinite dimensional non-separable Banach space.

The present paper is organized into six sections, including the Introduction. Section \ref{section2} contains notations and preliminaries. In Section \ref{section3} we review some elementary properties of various notions of convergence in spaces of sets. In Sections \ref{section4} we establish a convexification in the limit result for Fisher convergent sequences of sets in $\fC(\bE)$. In Section \ref{section5} we introduce a notion of expectation for random sets with values in $\fC(\bE)$, where the underlying Banach space $\bE$ is not necessarily separable. Finally in Section \ref{section6} we present two versions of the SLLN, with Fisher-type convergence, for i.i.d. random sets with values in $\fC_c(\bE)$ and $\fC(\bE)$ respectively.

\section{Notation and preliminaries}\label{section2}

Let $M$ be a metric space with distance $d_M$. For $Z\subset M$, we denote by $B_M(a,r)$ and $B_M[a,r]$ an open and a closed ball in $M$ with centre $a$ and radius $r$. For any $Z\subset M$, $\overline Z$ stands for the closure of $Z$ in $M$. If $a\in M$ and $Z$ is a nonempty subset of $M$, we set
$d_M(a,Z)=\inf\{ d_M(a,z)\; |\; z\in Z \}$. The Pompeiu-Hausdorff distance $h_M(X,Y)$ between two nonempty bounded sets $X,Y\subset M$ is defined by
\begin{equation}\nonumber
h_M(X,Y)=\max \{ e_M(X,Y), e_M(Y,X)\}
\end{equation}
where $e_M(X,Y)=\sup\{ d_M(x,Y) \;|\;x\in X \}$ and $e_M(Y,X)=\sup\{ d_M(y,X) \;|\;y\in Y \}$.

Throughout the present paper $\bE$ is a real Banach space with norm $\| \cdot\|$. The open (resp. closed) unit ball with centre at the origin of $\bE$ are denoted by $B_{\bE}$ (resp. $\overline{B_{\bE}}$). Moreover $coX$ and $\overline{co}X$ stand for the convex hull and the closed convex hull of a set $X\subset \bE$.  For any nonempty sets $X,Y\subset \bE$ and $\l \geq0$ the {\it sum} $X+Y$ and the {\it product} $\l X$ are given by
\begin{equation*}
 X+Y=\{x+y\;|\;x\in X, y\in Y\},\qquad\qquad \l X=\{\l X\;|\; x\in X\}.
\end{equation*}

We define
\begin{gather*}
\fB(\bE) = \{ X\subset \bE \; |\; X \; \text{is nonempty bounded} \} \\
\fC(\bE) = \{ X\subset \bE \; |\; X \; \text{is nonempty closed bounded} \} \\
\fC_{c}(\bE) = \{ X\subset \bE \; |\; X \; \text{is nonempty closed convex bounded} \}.
\end{gather*}
It is worth noting that in each of the above spaces the Pompeiu-Hausdorff distance $h_{\bE}(X,Y)$ between two sets $X,Y$ is given, equivalently, by
\begin{equation*}
 h_{\bE}(X,Y) = \inf_{r>0}\{X\subset Y+rB, Y\subset X+rB\}.
\end{equation*}
The spaces $\fC(\bE)$ and $\fC_{c}(\bE)$ are equipped with the Pompeiu-Hausdorff distance $h_{\bE}$, under which each one of them is a complete metric space. On the other hand, $h_{\bE}$ is a pseudometric on $\fB(\bE)$ since $h_{\bE}(X,Y)=0$, where $X,Y\in \fB(\bE)$, does not imply $X=Y$. For $X\in\fB(\bE)$ we set $\|X\|=\sup\{\|x\|\;|\;x\in X\}$.

For a sequence $\{X_n\}_{n=1}^{+\infty}\subset \fB(\bE)$ and $X\in\fB(\bE)$ we shall consider the following notions of convergence.
\begin{definition}\label{def2.1}
The sequence $\{X_n\}_{n=1}^{+\infty}$ is said to converge to $X$ in the \textit{sense of Pompeiu-Hausdorff} (we write $X_n \xrightarrow{h}X$) if $\lim_{n\to+\infty}h(X_n,X)=0$.
\end{definition}

\begin{definition}\label{def2.2}
The sequence $\{X_n\}_{n=1}^{+\infty}$ is said to converge to $X$ in the \textit{sense of Fisher} (we write $X_n \xrightarrow{F}X$) if
\begin{itemize}
 \item[$(a_1)$] $\lim_{n\to+\infty}e_{\bE}(X_n,X)=0$
\item[$(a_2)$] $\lim_{n\to+\infty}d_{\bE}(z,X_n)=0$ for every $z\in \bE$.
\end{itemize}
\end{definition}

\begin{definition}\label{def2.3}
 The sequence $\{ X_n\}_{n=1}^{+\infty}$ is said to converge to $X$ in the \textit{sense of Wijsman} (we write $X_n \xrightarrow{W}X$) if
\begin{equation*}
 \lim_{n\to +\infty}d_{\bE}(z,X_n)=d_{\bE}(z,X) \qquad \text{for every}\; z\in X.
\end{equation*}
\end{definition}

\begin{Remark}\label{remark1}
 Let $\{ X_n\}_{n=1}^{+\infty}\subset \fB(\bE)$ and $X\in \fB(\bE)$ be given. If $\g$ denotes any one of the above modes of convergence, \textit{i.e.}, convergence in the sense of Pompeiu-Hausdorff, Fisher or Wijsman, then we have
\begin{equation*}
 X_n\xrightarrow{\g}X \Longleftrightarrow \overline{X}_n\xrightarrow{\g}\overline{X} \Longleftrightarrow X_n\xrightarrow{\g}\overline{X} \Longleftrightarrow \overline{X}_n\xrightarrow{\g}X .
\end{equation*}
\end{Remark}

In what follows, when the role of the underlying metric space is evident and clarity is not affected, we will drop subscripts and write $d, e, h, B(a,r), B, \dots$ instead of $d_{\bE}, e_{\bE}, h_{\bE}, B_{\bE}(a,r), B_{\bE},\dots$.

\section{Auxiliary results}\label{section3}

In this section we review a few elementary properties of convergent sequences of sets which will be useful in what follows. Some of them could be easily deduced from results contained in \cite{Beer1993}. Yet, for the sake of completeness proofs are included.

\begin{proposition}\label{proposition3.1}
 Let $\{X_n\}_{n=1}^{+\infty}\subset \fC(\bE)$ and $X\in\fC(\bE)$. Then $X_n \xrightarrow{W}X$ if and only if:
\begin{itemize}
 \item[(i)] for each $z\in X$, $\lim_{n\to+\infty}d(z,X_n)=0$ (equivalently there exists a sequence $\{x_n\}_{n=1}^{+\infty}\subset \bE$, with $x_n\in X_n$, such that $x_n\to z$ as $n\to+\infty$).
 \item[(ii)] for each $z\notin X$, given $0<\ep<d(z,X)$, there exists $n_0\in\bN$ such that
\begin{equation*}
 X_n\cap B(z,d(z,X)-\ep)=\emptyset \quad \text{for all} \;\; n\geq n_0.
\end{equation*}
\end{itemize}
\end{proposition}
\begin{proof}
 Suppose that $X_n \xrightarrow{W}X$. Then $(i)$ is obvious. Arguing by contradiction, suppose that $(ii)$ is not valid. Then for some $z\notin X$ there exist $0<\ep<d(z,X)$ and a subsequence $\{X_{n_k}\}_{k=1}^{+\infty}$ such that $X_{n_k}\cap B(z,d(z,X)-\ep)\neq\emptyset$ for all $k\in \bN$. Hence
$d(z,X_{n_k})< d(z,X)-\ep$ for all $k\in \bN$, a contradiction, as $d(z,X_{n})\to d(z,X)$ for $n\to+\infty$.

Suppose that $(i)$ and $(ii)$ are satisfied. We want to show that, for every $z\in\bE$
\begin{equation}\label{eq3.1}
 \lim_{n\to+\infty} d(z,X_{n})= d(z,X)
\end{equation}
This is obvious if $z\in X$, in view of $(i)$. Let $z\notin X$ and, arguing by contradiction, suppose that \eqref{eq3.1} does not hold. Then there exist $0<\ep<d(z,X)$ and a subsequence $\{X_{n_k}\}_{k=1}^{+\infty}$ for which one of the following is valid:
\begin{eqnarray}
d(z,X_{n_k})\leq d(z,X)-\ep \qquad \text{for every}\; k\in\bN \label{eq3.2}\\
d(z,X_{n_k})\geq d(z,X)+\ep \qquad \text{for every}\; k\in\bN \label{eq3.3}
\end{eqnarray}
Consider \eqref{eq3.2}. By virtue of $(ii)$ there exists $n_0 \in \bN$ such that $X_n\cap B(z,d(z,X)-\ep/2)=\emptyset$ for all $n_{k_0}\geq n_0$. Hence $d(z,X_{n})< d(z,X)-\ep/2$ for every $n\geq n_0$ and thus for some $k_0\in\bN$ with $n_{k_0}\in\bN$ we have
\begin{equation*}
 d(z,X_{n_k})\geq d(z,X)-\ep/2 \qquad \text{for every}\; k\geq k_0,
\end{equation*}
a contradiction to \eqref{eq3.2}.

Consider \eqref{eq3.3}. Clearly $d(z,X_{n_k})> d(z,X)+\ep/2$ for every $k\in\bN$ and thus
\begin{equation*}
 X_{n_k}\cap B(z,d(z,X)+\ep/2)=\emptyset \qquad \text{for every}\; k\in\bN.
\end{equation*}
Now fix $x\in X$ so that $\|z-x\|< d(z,X)+\ep/2$. In view of $(i)$ there exists a sequence $\{x_n\}_{n=1}^{+\infty}$, with $x_n\in X_n$, such that $x_n\to x$ an $n\to +\infty$. Hence for $n$ large enough, say $n\geq n_0$, we have $\|z-x_n\|< d(z,X)+\ep/2$ and thus, a fortiori, $ d(z,X_{n})< d(z,X)+\ep/2$ for every $n\geq n_0$. Taking $k_0\in\bN$ so that $n_{k_0}\geq n_0$, we have that
\begin{equation*}
 d(z,X_{n_k})< d(z,X)+\ep/2 \qquad \text{for every}\; k\in\bN,
\end{equation*}
a contradiction to \eqref{eq3.3}. Consequently \eqref{eq3.1} holds thus completing the proof.
\end{proof}

\begin{Remark}
 In view of Proposition \ref{proposition3.1} we have that if a sequence $\{X_n\}_{n=1}^{+\infty}\subset \fC(\bE)$ converges in the sense of Wijsman to $X\in\fC(\bE)$, then this limit is unique.
\end{Remark}

\begin{proposition}\label{proposition3.2}
 Let $\{ X_n\}_{n=1}^{+\infty}\subset \fC(\bE)$ and let $X\subset\fC(\bE)$. Then
\begin{equation*}
 X_n\xrightarrow{h} X \Rightarrow X_n\xrightarrow{F} X \Rightarrow X_n\xrightarrow{W} X.
\end{equation*}
\end{proposition}
\begin{proof}
Suppose that $X_n\xrightarrow{F} X$. Then both conditions $(a_1)$ and $(a_2)$ of Definition \ref{def2.2} are trivially satisfied since, from the assumption, $e(X_n,X)\to 0$ and $e(X,X_n)\to 0$ as $n\to\bN$. Hence $X_n\xrightarrow{F} X$.

Suppose that $X_n\xrightarrow{F} X$. To prove that $X_n\xrightarrow{W} X$ it suffices to show that both conditions $(i)$ and $(ii)$ of Proposition \ref{proposition3.1} are satisfied. It is obvious that $(i)$ holds. Arguing by contradiction, suppose that $(ii)$ does not. Then, for some $z\notin X$ and some $0<\ep<d(z,X)$, there is a subsequence $\{ X_{n_k}\}_{k=1}^{+\infty}$ such that $X_{n_k}\cap B(z,d(z,X)-\ep)\neq\emptyset$ for every $k\in\bN$, and thus
\begin{equation}\label{eq3.4}
 d(z,X_{n_k})< d(z,X)-\ep \qquad \text{for every}\; k\in\bN,
\end{equation}
As $e(X_{n_k},X)\to 0$ for $k\to+\infty$, there exists $k_0\in\bN$ such that $X_{n_{k_0}}\subset X+\ep B/2$. By the latter and \eqref{eq3.4} it follows that
\begin{equation*}
 d(z,X)-\ep/2 \leq d(z,X+\ep B/2)\leq d(z,X_{n_{k_0}})<d(z,X)-\ep,
\end{equation*}
a contradiction. Hence also $(ii)$ is valid and thus $X_n\xrightarrow{W} X$ thus completing the proof.
\end{proof}

\begin{Remark}
The following examples show that the Pompeiu-Hausdorff convergence is stronger than Fisher's convergence and that the latter is in turn stronger than Wijsman's convergence.

Let $\bH$ be a real infinite dimensional Hilbert space with inner product $\langle\cdot,\cdot\rangle$ and induced norm $\|\cdot\|$. Let $S=\{x\in\bH\;|\;\|x\|=1\}$ and  $\overline{B}=\{x\in\bH\;|\;\|x\|\leq 1\}$. Let us show that for each $0<r<1$ there exists an infinite sequence $\{e_n\}\subset S$ such that
\begin{equation}\label{eq3.5}
 \|e_n-e_m\|\geq r \qquad\qquad\qquad \text{for all}\; n,m\in\bN, m\neq n.
\end{equation}
Let $e_1\in S$. Evidently $S\setminus \overline{B}(e_1,r) \neq \emptyset$ and thus there exists an element $e_2\in S$ such that $\|e_2-e_1\| >r$. Observe that
\begin{equation}\label{eq3.6}
 S\setminus (\overline{B}(e_1,r)\cup\overline{B}(e_2,r))\neq\emptyset
\end{equation}
In the contrary case $S\subset \{e_1,e_2\} +r\overline{B}\subset  \text{co}\{e_1,e_2\} +r\overline{B}$ and thus $\overline{B}\subset  \text{co}\{e_1,e_2\} +r\overline{B}$ which, by R\text{\r{a}}dstr$\Ddot{\text o}$m cancellation law \cite{Radstrom1952}, implies that $(1-r)\overline{B}\subset  \text{co}\{e_1,e_2\}$. This is a contradiction, as $\bH$ is infinite dimensional. Since \eqref{eq3.6} holds, there exists an element, say $e_3\in S$, such that $\|e_3-e_1\| >r$ and $\|e_3-e_2\| >r$. Thus, by induction, one can construct a sequence $\{e_n\}_{n=1}^{+\infty}\subset S$ satisfying \eqref{eq3.5}.
\begin{example}
The sequence $\{X_n\}_{n=1}^{+\infty}\subset \fC_c(\bH)$ defined by $X_n=\text{co}\{e_1,\dots,e_n\}$, with $n\in\bN$, satisfies
\begin{equation}\label{eq3.7}
 X_n\xrightarrow{F}X\qquad\text{and}\qquad X_n \overset{h}\nrightarrow X
\end{equation}
where $X=\cco E$ and $E=\{e_n\}_{n=1}^{+\infty}$.

Clearly $X_n\subset X_{n+1}\subset X$, $n\in\bN$, and hence $e(X_n,X)=0$ for every $n\in\bN$. Moreover,
\begin{equation*}
 \lim_{n\to+\infty} d(z,X_n)=0\qquad\text{for each}\qquad z\in X,
\end{equation*}
because, given $\ep>0$, for some $n_0\in\bN$ we have $d(z,X_{n_0})<\ep$ and thus, a fortiori, $d(z,X_{n})<\ep$ for all $n\geq n_0$. Therefore $X_n \xrightarrow{F}X$. On the other hand, if $X_n\xrightarrow{h}X$ then, denoting by $\a$ the Kuratowski measure of noncompactness \cite{kuratowski1966}, we have $\lim_{n\to+\infty} \a(X_n)=\a(X)$. As $\a(X)=0$ for each $n\in\bN$, it follows that $\a(X)=0$, \textit{i.e.}, $X$ is compact. This impossible since, in view of \eqref{eq3.5}, the sequence $\{e_n\}_{n=1}^{+\infty}\subset X$ does not contain convergent subsequences. Consequently $X_n \overset{h}\nrightarrow X$.
\end{example}

\begin{example}
 Let $E=\{e_n\}_{n=1}^{+\infty}\subset S$ be as in the previous example. Thus $\{e_n\}$ satisfies \eqref{eq3.5} for some $0<r<1$. Fix $1<\l<2/\sqrt{4-r^2}$. Then the sequence $\{X_n\}_{n=1}^{+\infty}\subset \fC_c(\bH)$ defined by $X_n=\cco\{\l e_n, \overline{B}\}$, $n\in\bN$, satisfies
\begin{equation}\label{eq3.8}
  X_n\xrightarrow{W}\overline{B}\qquad\text{and}\qquad X_n \overset{F}\nrightarrow \overline{B}.
\end{equation}
Let us show that $X_n\xrightarrow{W}\overline{B}$. Set $\rho=\sqrt{\l^2-1}$. Clearly
\begin{equation}\label{eq3.9}
B(\l e_n, \rho)\cap\overline{B}\neq\emptyset \qquad\qquad\text{for every}\;n\in\bN,
\end{equation}
since $\|\l e_n-e_n\| = \l-1<\rho$. Moreover,
\begin{equation}\label{eq3.10}
B\left[\l e_i, \rho\right]\cap B\left[\l e_j, \rho\right] =\emptyset \qquad\qquad\text{for every}\;i\neq j \in\bN,
\end{equation}
because for arbitrary $z\in B\left[\l e_i, \rho\right]$ and $z^{\prime}\in B\left[\l e_j, \rho\right]$ one has
\begin{equation*}
 \|z-z^{\prime}\| = \|\l(e_i-e_j)+(z-\l e_i)-(z^{\prime}-\l e_j)\|\geq \l\|e_i-e_j\|-2\rho \geq \l r-2\sqrt{\l^2-1}>0.
\end{equation*}
Furthermore, let us prove that for every $e_n\in E$
\begin{equation}\label{eq3.11}
 y\in S\setminus  B\left[\l e_n, \rho\right]\quad\Rightarrow\quad X_n\subset\{x\in\bH\;\|\; \langle x-y,y\rangle \leq 0\}.
\end{equation}
Indeed, let $y\in S$ satisfy $\|y-\l e_n\|>\rho$. This implies that $ \langle\l e_n,y\rangle<1$. Since for each $t\in [0,1]$ and $z\in\overline{B}$ ,
\begin{equation*}
 \langle(1-t)\l e_n +tz-y,y\rangle=(1-t) \langle\l e_n,y\rangle+ t\langle z,y\rangle-\| y\|^2\leq (1-t)+t-1=0,
\end{equation*}
it follows that each point of $X_n$ is contained in the half space $\{x\in\bH\;|\; \langle x-y,y\rangle\leq 0\}$, and thus \eqref{eq3.11} is valid.

We are now ready to prove that for every $z\in\bH$
\begin{equation}\label{eq3.12}
 \lim_{n\to+\infty} d(z,X_n) = d(z,\overline{B}).
\end{equation}
Let $z\in\bH\setminus \overline{B}$ be arbitrary (if $z\in\overline{B}$ there is nothing to prove) and denote by $y$ the point at which the closed linear segment $tz$ ($0\leq t\leq 1$) meets $S$. In view of \eqref{eq3.10} one (and only one) of the following cases occurs:
\begin{itemize}
 \item[$(j)$] $y\in B\left[\l e_k,\rho\right]$ for some $e_k\in E$ (any such $e_k$ is unique by \eqref{eq3.10}.
\item[$(jj)$]  $y\notin B\left[\l e_n,\rho\right]$ for every $n\in\bN$.
\end{itemize}
Consider $(j)$. Then by \eqref{eq3.10}, $y\notin B\left[\l e_n,\rho\right]$ for every $n\neq k$ and hence by \eqref{eq3.11}
\begin{equation*}
 X_n \subset \{x\in\bH \;|\; \langle x-y,y\rangle\leq 0\}\qquad\qquad \text{for every}\; n\neq k.
\end{equation*}
The latter implies that $d(z,X_n)\geq \|z-y\|$ since $y\in X_n$. As $d(z,\overline{B})=\|z-y\|$, it follows that
\begin{equation*}
 d(z,X_n)=d(z,\overline{B})\qquad \qquad \text{for every}\; n\neq k,
\end{equation*}
and thus \eqref{eq3.12} holds.\\
Consider $(jj)$. Then by \eqref{eq3.11},
\begin{equation*}
 X_n \subset \{x\in\bH \;|\; \langle x-y,y\rangle\leq 0\}\qquad\qquad \text{for every}\; n\in\bN
\end{equation*}
and as before one can show that \eqref{eq3.12} holds. Therefore $X_n\xrightarrow{W}\overline{B}$. On the other hand $X_n \overset{F}\nrightarrow \overline{B}$ because, for each $n\in\bN$, we have $e(X_n,\overline{B})=\|\l e_n-e_n\| = \l-1>0$. This completes the proof of \eqref{eq3.7}.
\end{example}
\end{Remark}

\begin{proposition}\label{proposition3.3}
 Let $\{X_n\}_{n=1}^{+\infty},\{Y_n\}_{n=1}^{+\infty}\subset\fB(\bE)$ satisfy $X_n\xrightarrow{W}X$ and $Y_n\xrightarrow{W}Y$ for some $X,Y\in\fB(\bE)$. Let $\{\l_n\}_{n=1}^{+\infty},\{\m_n\}_{n=1}^{+\infty}\subset [0,+\infty)$ be such that $\l_n\to\l$ and $m_n\to\m$ for some $\l,\m\geq 0$. Then
\begin{equation}\label{eq3.13}
 \overline{\l_n X_n + \m_n Y_n}\quad \xrightarrow{F}\quad \overline{\l X + \m Y}.
\end{equation}
\end{proposition}
\begin{proof}
 Let $\ep>0$ be arbitrary. From the assumption, there exists $n_0\in\bN$ such that $X_n\subset X+\ep B$ and $Y_n\subset Y+\ep B$ for every $n\geq n_0$. Then for all $n$ large enough we have
\begin{equation*}
 \l_n X_n + \m_n Y_n\subset \l_n X + \m_n Y +\ep (\l_n+\m_n)B\subset \l X + \m Y +\ep(\l+\m+1)B
\end{equation*}
which shows that
\begin{equation}\label{eq3.14}
 \lim_{n\to+\infty} e(\l_n X_n + \m_n Y_n, \l X + \m Y)=0.
\end{equation}
Let $z\in \l X + \m Y$ be arbitrary, thus $z=\l x+\m y$ for some $x\in X$ and $y\in Y$. As $X_n \xrightarrow{F}X$ and $Y_n \xrightarrow{F}Y$, there exists two sequences $\{x_n\}_{n=1}^{+\infty},\{y_n\}_{n=1}^{+\infty}$ with $x_n\in X_n$ and $y_n\in Y_n$ such that $x_n\to x$ and $y_n\to y$ as $n\to+\infty$. Set $z_n= \l_n x_n+\m_n y_n$ and observe that $z_n\in  \l_n X_n + \m_n Y_n$ and $z_n\to z$. Clearly $d(z,\l_n X_n+\m_n Y_n)\leq\| z-z_n\|$ and thus
\begin{equation}\label{eq3.15}
 \lim_{n\to+\infty} d(z,\l_n X_n+\m_n Y_n)=0\qquad\qquad \text{for every}\; z\in X+Y.
\end{equation}
The statement \eqref{eq3.13} is an immediate consequence of \eqref{eq3.14}, \eqref{eq3.15} and Remark \ref{remark1}. This completes the proof.
\end{proof}

\begin{proposition}\label{proposition3.4}
 Let $\{ X_n\}_{n=1}^{+\infty}, \{ Y_n\}_{n=1}^{+\infty}\subset \fC(\bE)$ and let $X, Y\subset\fC(\bE)$. Suppose that $X_n\xrightarrow{W} X$ and $Y_n\xrightarrow{W} Y$ as $n\to+\infty$ and $h(X_n,Y_n)<r$ for every $n\in \bN$, $r>0$. Then, $h(X,Y)\leq r$.
\end{proposition}

\begin{proof}
 Arguing by contradiction, suppose that for some $\theta>0$, $h(X,Y)> r+\theta$ and to fix ideas let $e(X,Y)>r+\theta$ (the argument is similar if $e(Y,X)>r+\theta$). Fix an $x\in X$ satisfying
\begin{equation}\label{eq3.16}
d(x,Y)>r+\theta.
\end{equation}
Now $x\in X$ and $X_n\xrightarrow{W} X$, hence, by Proposition \ref{proposition3.1}, there exists a sequence $\{x_n \}$, with $x_n\in X_n$, such that $x_n\to x$ as $n\to+\infty$. On the other hand, $x\notin Y$ by \eqref{eq3.16} and $Y_n\xrightarrow{W} Y$ and thus, by Proposition \ref{proposition3.1}, given $0<\ep<\theta/3$ there exists $n_1 >n_0$ such that
\begin{equation}\label{eq3.17}
Y_n\cap B(x, d(x, Y)-\ep )=\emptyset \quad \text{for every} \;\; n\geq n_0 .
\end{equation}
From \eqref{eq3.16}, as $x_n\to x$, it follows that there exists $n_1\geq n_0$ such that 
\begin{equation}\label{eq3.18}
 d(x_n ,Y)>r+\theta \quad \text{for every} \;\; n\geq n_1 .
\end{equation}
Take $n_2 >n_1$ such that $\|x_n -x\|\leq \ep$ and $|d(x_nY)-d(x,Y)| <\ep$ for every $n_1\geq n_2$. Hence
\begin{equation*}
B(x_n, d(x_n, Y)-3\ep )\subset B(x, d(x, Y)-\ep ) \quad \text{for every} \;\; n\geq n_2 ,
\end{equation*}
and thus by \eqref{eq3.17}
\begin{equation}\label{eq3.19}
Y_n\cap B(x_n, d(x_n, Y)-3\ep )=\emptyset \quad \text{for every} \;\; n\geq n_2 .
\end{equation}
Now fix $n\geq n_2$. Then by virtue of \eqref{eq3.19} and \eqref{eq3.18} one has
\begin{equation*}
 d(x_n ,Y_n)\geq d(x_n, Y)-3\ep >r+\theta -3\ep.
\end{equation*}
Since on the other hand $d(x_n ,Y_n) \leq e(X_n, Y_n)\leq h(X_n, Y_n)< r$, it follows that $r>r+\theta-3\ep$. This is a contradiction, for $0<\ep<\theta/3$. Therefore $h(X,Y)\leq r$, which completes the proof.
\end{proof}

\begin{proposition}\label{proposition3.5}
Let $\cal{Z}\subset \fC(\bE)$ be a nonempty compact set in the Pompeiu-Hausdorff metric $h$ of $\fC(\bE)$. Let $\{ Z_n\}_{n=1}^{+\infty}\subset \cal{Z}$ and let $C\in \fC_c(\bE)$ be compact. Then,
\begin{equation*}
Z_n \xrightarrow{W}C \quad\Rightarrow \quad Z_n \xrightarrow{h}C. 
\end{equation*}
\end{proposition}

\begin{proof}
It suffices to show that, given $\ep >0$, there exist $n^{\prime},n^{\prime\prime}\in\bN$ such that the following properties $(j)$ and $(jj)$ are satisfied:
\begin{itemize}
\item[$(j)$] $\qquad C\subset Z_n +\ep B$ \qquad\qquad\qquad for every $n\geq n^{\prime}$
\item[$(jj)$]  $\qquad Z_n \subset C +\ep B$ \qquad\qquad\qquad for every $n\geq n^{\prime\prime}$
\end{itemize}
Consider $(j)$. By hypothesis $Z_n \xrightarrow{W}C$ and hence, by Proposition \ref{proposition3.1}, for each $u\in C$ there exist $n_{u}\in\bN$ such that $d(u, Z_n)<\ep/2$ for every $n\geq n_u$. Since $d(x,Z_n)\leq \|x-u\| +d(u,Z_n)$, it follows that $d(x, Z_n)<\ep$ for every $x\in B(u,\ep/2)$ and $n\geq n_{u}$. Now, $\{B(u,\ep/2)\}_{u\in C}$ is an open covering of $C$, a compact set, hence it admits a finite subcovering, \textit{i.e.}, for some $u_1,\dots,u_d\in C$ we have
\begin{equation*}
C \subset B(u_1,\frac{\ep}{2})\cup \dots \cup B(u_d,\frac{\ep}{2}).
\end{equation*}
Set $n^{\prime}= \max \{ n_{u_1}, \dots,n_{u_d} \}$. Since each $x\in C$ is in some ball $B(u_{i},\ep/2)$, with $1\leq i\leq d$, we have $d(x,Z_n)<\ep$ for every $n\geq n^{\prime}$. Consequently $C \subset Z_n +\ep B$ for every $n\geq n^\prime$ and $(j)$ holds.

Consider $(jj)$. Arguing by contradiction, suppose that $(jj)$ does not hold. Then there exists a subsequence $\{ Z_{n_k}\}_{k=1}^{+\infty}$ such that
\begin{equation}\label{eq3.20}
 Z_{n_k}\nsubset C +\ep B \qquad \text{for every}\;\;k\in\bN.
\end{equation}
As $\{ Z_{n_k}\}_{k=1}^{+\infty}\subset \cal{Z}$, where $\cal{Z}\subset \fC(\bE)$ is compact in the Pompeiu-Hausdorff metric $h$ of $\fC(\bE)$, there exists a subsequence $\{ Z_{n_{k_j}}\}_{j=1}^{+\infty}$ and a set $Z\in\cal{Z}$ such that $Z_{n_{k_j}} \xrightarrow{h} Z$ as $j\to+\infty$.
From this and \eqref{eq3.20} it follows that $Z\nsubset C +(\ep/2) B$ and hence, for some $z\in Z$, we have $d(z,C)\geq\ep/2$. Since $z\in Z$ and $Z_{n_{k_j}} \xrightarrow{h} Z$, there exists $j_0\in\bN$ such that
\begin{equation}\label{eq3.21}
 Z_{n_{k_j}} \cap B(z,\frac{\ep}{8}) \neq \emptyset \qquad \text{for all} \;\; j\geq j_0 .
\end{equation}
On the other hand, $z\notin C$ and $Z_{n_{k_j}} \xrightarrow{W} C$, and thus by Proposition \ref{proposition3.1} there exists $j_1\in\bN$ such that
\begin{equation}\label{eq3.22}
 Z_{n_{k_j}} \cap B(z,d(z,C)-\frac{\ep}{4})=\emptyset \qquad \text{for all} \;\;j\geq j_1 .
\end{equation}
Let $j=\max\{j_0,j_1\}$. Then by \eqref{eq3.21} and \eqref{eq3.22}
\begin{equation*}
 d(z, Z_{n_{k_j}})< \frac{\ep}{8} \quad \text{and}\quad d(z, Z_{n_{k_j}})\geq d(z,C)-\frac{\ep}{4},
\end{equation*}
which yields a contradiction since $d(z,C)\geq \ep/2$. Hence also $(jj)$ holds and this completes the proof.
 \end{proof}

For $\cal Z$ a nonempty subset of $\fC_c(\bE)$ define
\begin{equation*}
 \co \cZ =\{ Y \in\fC_c(\bE) \;|\; Y=\overline{\sum^n_{i=1} \lambda_i X_i },\;
\text{where}\; X_i\in\cZ, \lambda_i\in [ 0,1 ] \; \text{with} \; \sum_{i=1}^{d}\lambda_i=1, n\in\bN \}.
\end{equation*}
Further, we denote by $\text{cl}_{\fC_c(\bE)}[\co\cal{Z}]$ the closure of $\co\cal{Z}$ in the Pompeiu-Hausdorff metric of $\fC_c(\bE)$. The set $\co\cZ$ is convex, \textit{i.e.}, $Y_1,Y_2\in\co\cZ$ implies that $\overline{(1-\l)Y_1+\l Y_2}\in \co\cZ$ for every $\l\in [0,1]$. Clearly the set $\text{cl}_{\fC_c(\bE)}[\co\cal{Z}]$ is also convex.

\begin{proposition}[(Mazur's Theorem)]\label{proposition3.6}
 Let $\cZ$ be a nonempty compact subset of $\fC_c(\bE)$ in the Pompeiu-Hausdorff metric of $\fC_c(\bE)$. Then the set $\text{cl}_{\fC_c(\bE)}[\co\cal{Z}]$ is compact and convex.
\end{proposition}
\begin{proof}
 It suffices to show that $\text{cl}_{\fC_c(\bE)}[\co\cal{Z}]$ is compact. Let $\ep >0$. By hypothesis $\cal{Z}$ is compact and thus it admits a finite $\ep/2$-net $\{A_1,\dots,A_p\} \subset \cal{Z}$. Let $\D =\{ (\a_1,\dots, \a_p)\in \bR^p \;|\; \a_i\in [ 0,1 ],  \a_1+\dots +\a_p=1\}$. Define $\Phi\;:\;\D\to\fC_c(\bE)$ by
\begin{equation*}
 \Phi(\a_1,\dots,\a_p) = \overline{\sum_{i=1}^p\a_i A_i} \qquad\qquad  (\a_1,\dots, \a_p)\in\D,
\end{equation*}
and set $\cR=\Phi(\D)$. Evidently $\cR$ is compact and thus it admits a finite $\ep/2$-net $\{ Y_1,\dots,Y_q \}\subset \cR$.
We want to show that $\{ Y_1,\dots,Y_q \}$ is an $\ep$-net of $\co\cZ$. Indeed let $X\in\co\cZ$ be arbitrary, and thus
\begin{equation*}
X=\overline{\sum^n_{i=1} \lambda_i X_i },                                                                                                                      
\end{equation*}
for some $X_i\in\cZ$, $\l_i\in [0,1]$ with $\l_1 +\dots +\l_n =1$. As $\{A_1,\dots,A_p\}$ is an $\ep/2$-net of $\cZ$, one has that each $X_i$, $1\leq i\leq n$, is contained in some ball, say $B(A_i^{\prime},\ep/2)$, for some $A_i^{\prime}\in \{A_1,\dots,A_p\}$. Set
\begin{equation*}
Y=\overline{\sum^n_{i=1} \lambda_i A_i^{\prime} }.                                                                                                                       
\end{equation*}
Clearly $h(X,Y)<\ep/2$. Moreover, $Y\in\cR$ because for some $\a_i\in [0,1]$ with $\a_1 +\dots +\a_p =1$ we have 
\begin{equation*}
Y=\overline{\sum^p_{i=1} \a_i A_i }.                                                                                                                       
\end{equation*}
Since $\{Y_1,\dots,Y_q \}$ is an $\ep/2$-net of $\cR$ and $Y\in\cR$, it follows that $Y\in B(Y_j,\ep/2)$ for some $1\leq j\leq q$. Whence $X\in B(Y_j,\ep/2)$, for $h(X,Y_j)\leq h(X,Y) +h(Y,Y_j)<\ep$. As $X\in\co\cZ$ is arbitrary, we conclude that $\{Y_1,\dots,Y_q\}$ is an $\ep$-net of $\co\cZ$ and hence also for $\text{cl}_{\fC_c(\bE)}[\co\cal{Z}]$. Hence $\text{cl}_{\fC_c(\bE)}[\co\cal{Z}]$ is totally bounded. Furthermore $\text{cl}_{\fC_c(\bE)}[\co\cal{Z}]$ is complete since it is closed in $\fC_c(\bE)$, a complete metric space. Consequently $\text{cl}_{\fC_c(\bE)}[\co\cal{Z}]$ is compact, completing the proof.
\end{proof}

\section{Convexification in the limit results}\label{section4}
In this section we prove some convexification in the limit results which turn out to be useful for the proof of the SLLN for $\fC(\bE)$-valued random sets.

For any set $D\in \fC(\bE)$ and $n\in\bN$ we set
\begin{equation*}
 D[n]=\frac{\overbrace{D+\dots+D}^{n-\text{times}}}{n}
\end{equation*}

\begin{proposition}\label{proposition4.1}
 Let $C\in\ffC$. Then
\begin{equation}\label{eq4.1}
 \overline{C[n]}\xrightarrow{F}\cco C \qquad \text{as}\quad n\to+\infty
\end{equation}
\end{proposition}
\begin{proof}
Denote by $\bQ^+$ the set of positive rationals $\rho\geq 0$. Define
\begin{equation*}
 Q = \{ x\in \co C \;|\; x=\sum_{i=1}^d \rr_i x_i ,\; \text{where}\; x_i\in C, \rr_i\in\bQ^+, \;\text{with}\;\sum_{i=}^d \rr_i =1, d\in\bN \}
\end{equation*}
and observe that $Q\in\fB(\bE)$ and $\overline{Q}=\cco C$.

Indeed let $z\in Q$ be arbitrary. Then, for some $x_1,\dots,x_d \in C$ and $p_1,\dots,p_d\in\bN$, we have
\begin{equation*}
 z = \frac{p_1}{p}x_1 +\dots + \frac{p_d}{p}x_d \qquad  \text{where}\qquad p_1+\dots +p_d =p.
\end{equation*}
Moreover
\begin{equation*}
 z\in C[n], \qquad \text{for every}\;\; n\in \cI =\{p, 2p,\dots,kp,\dots \}
\end{equation*}
because for each $n\in\cI$, say $n=kp$, we have
\begin{gather*}
 z= \frac{kp_1 x_1 +\dots + kp_d x_d}{kp}\\
=\frac{\overbrace{x_1 +\dots + x_1}^{kp_1 -\text{times}} + \dots + \overbrace{x_d +\dots + x_d}^{kp_d -\text{times}}}{kp}\\
\in \frac{\overbrace{C +\dots + C}^{kp -\text{times}}}{kp}=C[n].
\end{gather*}

{\bf Claim.} There exists a sequence $\{c_n \}_{n=1}^{+\infty}\subset \bE$ satisfying the following properties:
\begin{eqnarray}
c_n=z \qquad \text{if}\;\;  n= kp\;\;\text{for}\;\; k=1,2,\dots \label{eq4.2}\\
 c_n\in C[n] \qquad \text{for every}\;\; n\geq p \qquad\;\;\label{eq4.3}\\
c_n\rightarrow z \qquad \text{as}\;\; n\to+\infty\qquad\qquad\quad\; \label{eq4.4}
\end{eqnarray}
Indeed fix $c\in C$ and define the sequence $\{c_n\}_{n=1}^{+\infty}$ as follows
\[c_n =
 \begin{cases}
z \quad & \text{if} \;\;n=kp,\; k\in\bN\\
\frac{\overbrace{x_1 +\dots + x_1}^{kp_1 -\text{times}} + \dots + \overbrace{x_d +\dots + x_d}^{kp_d -\text{times}} + \overbrace{c +\dots + c}^{j -\text{times}}}{kp+j} & \text{if}\;\; n=kp+j,\;k\in\bN \\
 & \quad j=1,\dots, p-1.
\end{cases}
\]
Clearly for every $k\in\bN$ and $j=1, 2,\dots, p-1$ we have $c_{kp}\in C[kp]$ and $c_{kp+j}\in C[kp+j]$. It is evident that the sequence $\{c_n\}_{n=1}^{+\infty}$ satisfies \eqref{eq4.2} and \eqref{eq4.3}. Furthermore, for any $k\in\bN$ and $j=1, 2,\dots, p-1$, we have
\begin{gather*}
 c_{kp+j}-z=\frac{\overbrace{x_1 +\dots + x_1}^{kp_1 -\text{times}} + \dots + \overbrace{x_d +\dots + x_d}^{kp_d -\text{times}} + \overbrace{c +\dots + c}^{j -\text{times}}}{kp+j}\\
-\frac{\overbrace{x_1 +\dots + x_1}^{kp_1 -\text{times}} + \dots + \overbrace{x_d +\dots + x_d}^{kp_d -\text{times}}}{kp}\\
=\frac{kp}{kp+j} \frac{\overbrace{x_1 +\dots + x_1}^{kp_1 -\text{times}} + \dots + \overbrace{x_d +\dots + x_d}^{kp_d -\text{times}}}{kp} + \frac{j}{n+j}c\\
-\frac{\overbrace{x_1 +\dots + x_1}^{kp_1 -\text{times}} + \dots + \overbrace{x_d +\dots + x_d}^{kp_d -\text{times}}}{kp}
=\left( \frac{kp}{kp+j}-1\right) z+ \frac{j}{kp+j}c
\end{gather*}
From the latter, letting $k\to+\infty$ it follows that $c_{kp+j}\to z$. This and \eqref{eq4.2}imply that $c_n\to z$ as $n\to+\infty$. Hence also \eqref{eq4.2} is satisfied and the Claim is proved.

Now we are in position to complete the proof of Proposition \ref{proposition4.1}. Evidently $e(C[n],\cco C)=0$, since $C[n]\subset \cco C$ for each $n\in\bN$. Moreover, by the Claim $\lim_{n\to+\infty} d(z,C[n])=0$. As $z\in Q$ is arbitrary, it follows that $C[n]\xrightarrow{F}Q$ as $n\to+\infty$. Hence $\overline{C[n]}\xrightarrow{F}\overline{Q}$, \textit{i.e.}, $\overline{C[n]}\xrightarrow{F}\overline{\co C}$ as $n\to+\infty$. Thus \eqref{eq4.1} is valid and the proof is complete.
\end{proof}

\begin{theorem}\label{theo4.1}
 Let $\cZ\subset \fC(\bE)$ be a nonempty compact set in the Pompeiu-Hausdorff metric $h$ of $\fC(\bE)$. Let $\{X_n\}_{n=1}^{+\infty}\subset \cZ$ and let $C\in\fC_c(\bE)$. Then
\begin{equation*}
 \frac{\overline{\overline{co}X_1 +\dots + \overline{co}X_n}}{n} \xrightarrow{F}C \qquad\text{as}\;n\to+\infty
\end{equation*}
implies
\begin{equation}\label{eq4.5}
 \frac{\overline{X_1 +\dots + X_n}}{n} \xrightarrow{F}C \qquad\text{as}\;n\to+\infty.
\end{equation}
\end{theorem}
\begin{proof}
The statement is proved if we show:
\begin{itemize}
\item[\textit{(j)}] $\lim_{n\to+\infty} e((X_1+\dots+X_n)/n, C) =0$
\item[\textit{(jj)}] $\lim_{n\to+\infty}d\left( z, (X_1 +\dots + X_n)/n\right)=0$ for each $z\in C$.
\end{itemize}
The fact that $(j)$ holds is obvious, since for each $n\in\bN$
\begin{equation*}
 e\left(\frac{X_1+\dots+X_n}{n}, C\right)\leq e\left(\frac{\cco X_1+\dots+\cco X_n}{n}, C\right),
\end{equation*}
and the right hand side vanishes as $n\to+\infty$.

Let us prove $(jj)$. Arguing by contradiction, suppose that $(jj)$ does not hold. Then for some $z\in C$ there exists $\theta >0$ and a subsequence $\{ (X_1 +\dots + X_{n_k})/n_k\}_{k=1}^{+\infty}$ such that
\begin{equation}\label{eq4.6}
d\left( z, \frac{X_1 +\dots + X_{n_k}}{n_k}\right)\geq \theta \qquad\text{for every}\;k\in\bN.
\end{equation}
Let $0<\ep <\theta/8$. Since $\cZ$ is compact in $\fC(\bE)$, it admits a finite $\ep$-net, say $\{C_1,C_2,\dots,C_d\}\subset \cZ$, for some $d\in\bN$. We now associate to the given sequence $\{X_n\}_{n=1}^{+\infty}\subset\cZ$ another sequence $\{C^{\prime}_n\}_{n=1}^{+\infty}$ constructed with the following procedure. Consider the ordered set of balls
\begin{equation*}
\{B(C_1,\ep),B(C_2,\ep),\dots,B(C_d,\ep)\}.
\end{equation*}
Then for any fixed $n\in\bN$, denote by $C^{\prime}_n$ the centre of the ball $B(C_i,\ep)$, where $i$ is the smallest index $1\leq i\leq d$ such that $X_n\in B(C_i,\ep)$. Clearly $h(C^{\prime}_n,X_n)<\ep$ for every $n\in\bN$. Furthermore, we associate to $\{ (X_1 +\dots + X_{n_k})/n_k\}_{k=1}^{+\infty}$ the following sequence:
\begin{equation}\label{eq4.7}
\left\lbrace \frac{C^{\prime}_1 +\dots + C^{\prime}_{n_k}}{n_k}\right\rbrace_{k=1}^{+\infty}.
\end{equation}

{\bf Claim 1.} There exists a subsequence $\{ (C^{\prime}_1 +\dots + C^{\prime}_{n_{k_j}})/n_{k_j}\}_{j=1}^{+\infty}$ of \eqref{eq4.7} and there is a set $\Gamma\in\fC_c(\bE)$ given by $\Gamma=\overline{\sum_{i=1}^d \lambda_i\cco C_i}$, for some $\lambda_1,\dots ,\lambda_d\geq 0$ with $\lambda_1+\dots +\lambda_d =1$, such that:
\begin{equation}\label{eq4.8}
\frac{\overline{C^{\prime}_1+\dots +C^{\prime}_{n_{k_j}}}}{n_{k_j}}\xrightarrow{F} \Gamma\qquad \text{as}\;j\to+\infty
\end{equation}
and
\begin{equation}\label{eq4.9}
\frac{\overline{\cco C^{\prime}_1+\dots +\cco C^{\prime}_{n_{k_j}}}}{n_{k_j}}\xrightarrow{h} \Gamma \qquad \text{as}\; j\to+\infty.
\end{equation}
Let us prove \eqref{eq4.8}. In view of the definition of the sequence $\{C^{\prime}_n\}_{n=1}^{+\infty}$, for every $k\in\bN$ there exists a partition of the set $\{1,2,\dots,n_k\}$ into $d$ pairwise disjoint sets $P^1_{n_k}, P^2_{n_k}, \dots,P^d_{n_k}$ consisting, respectively, of $p^1_{n_k}, p^2_{n_k}, \dots,p^d_{n_k}$ elements, where $0\leq p^i_{n_k}\leq n_k$ for $i=1,\dots,d$ and $p^1_{n_k}+\dots+p^d_{n_k}=n_k$, such that
\begin{equation*}
C^{\prime}_i=C_1\;\text{for}\; i\in P^1_{n_k},\; C^{\prime}_i=C_2\;\text{for}\; i\in P^2_{n_k},\;\dots\; ,C^{\prime}_i=C_d \;\text{for}\; i\in P^d_{n_k}.
\end{equation*}
Observe that $p^i_{n_k}=0$ whenever $P^i_{n_k}=\emptyset$.

To fix ideas, suppose that $p^i_{n_k}\geq 1$ for $i=1,\dots,d$ (if some $p^i_{n_k}=0$, the argument is similar). We have
\begin{multline}\label{eq4.10}
\frac{C^{\prime}_1+\dots +C^{\prime}_{n_k}}{n_k}=\frac{1}{n_k}\left[\overbrace{C_1+\dots + C_1}^{p^1_{n_k} -\text{times}} + \overbrace{C_2+\dots + C_2}^{p^2_{n_k} -\text{times}} + \dots + \overbrace{C_d+\dots + C_d}^{p^d_{n_k} -\text{times}}\right] \\
= \frac{1}{n_k}\left[ p^1_{n_k}\frac{\overbrace{C_1+\dots + C_1}^{p^1_{n_k} -\text{times}}}{p^1_{n_k}} + p^2_{n_k}\frac{\overbrace{C_2+\dots + C_2}^{p^2_{n_k} -\text{times}}}{p^2_{n_k}} + \dots + p^d_{n_k}\frac{\overbrace{C_d+\dots + C_d}^{p^d_{n_k} -\text{times}}}{p^d_{n_k}}\right] \\
=\frac{p^1_{n_k}}{n_k} C_1[p^1_{n_k}]+\frac{p^2_{n_k}}{n_k} C_2[p^2_{n_k}]+\dots +\frac{p^d_{n_k}}{n_k} C_d[p^d_{n_k}] .\qquad\quad
\end{multline}
For each fixed $i=1,\dots,d$ the sequence $\{ p^i_{n_k}\}_{k=1}^{+\infty}$ is non-decreasing, $i.e.$, $p^i_{n_k}\leq p^i_{n_{k+1}}$ for every $k\in\bN$. Consider the sequence
\begin{equation}\label{eq4.11}
\left\lbrace \left(\frac{p^1_{n_k}}{n_k},\frac{p^2_{n_k}}{n_k},\dots,\frac{p^d_{n_k}}{n_k}\right) \right\rbrace_{k=1}^{+\infty}.
\end{equation}
Evidently, $(p^1_{n_k}/n_k\,,\,p^2_{n_k}/n_k\,,\dots,\,p^d_{n_k}/n_k)\in \Delta$ for every $k\in\bN$, where
\begin{equation*}
\Delta= \{ (\l_1,\dots,\l_d) | \l_i\geq0,\;i=0,\dots,d\;\text{and}\; \l_1+\dots+\l_d=1\}, 
\end{equation*}
and thus the sequence \eqref{eq4.11} contains a subsequence
\begin{equation*}
\left\lbrace \left(\frac{p^1_{n_{k_j}}}{n_{k_j}},\frac{p^2_{n_{k_j}}}{n_{k_j}},\dots,\frac{p^d_{n_{k_j}}}{n_{k_j}}\right) \right\rbrace_{j=1}^{+\infty}.
\end{equation*}
which converges to a limit $(\l_1,\dots,\l_d)\in\Delta$ as $j\to+\infty$. For each fixed $1\leq i\leq d$, if $p^i_{n_{k_j}}\to+\infty$ as $j\to+\infty$, then by Propositions \ref{proposition4.1} and \ref{proposition3.3},
\begin{equation}\label{eq4.12}
 \frac{p^i_{n_{k_j}}}{n_{k_j}} \overline{C_i [p^i_{n_{k_j}}]} \xrightarrow{F} \l_i\cco C_i \qquad\qquad\text{as}\; j\to+\infty.
\end{equation}
It is evident that the latter limit remains valid also when $\{p^i_{n_{k_j}}\}_{j=1}^{+\infty}$ is bounded for, in this case, $\l_i=0$. Set $\Gamma = \overline{\sum_{i=0}^d \l_i\cco C_i}$. From \eqref{eq4.10}, by virtue of \eqref{eq4.12} and Proposition \ref{proposition3.3}, it follows that
\begin{equation*}
\frac{\overline{C^{\prime}_1+\dots +C^{\prime}_{n_{k_j}}}}{n_{k_j}}\xrightarrow{F} \Gamma \qquad \text{as}\;j\to+\infty
\end{equation*}
and thus \eqref{eq4.8} holds. Concerning \eqref{eq4.9}, it suffices to observe that
\begin{gather*}
\frac{\cco C^{\prime}_1+\dots +\cco C^{\prime}_{n_{k_j}}}{n_{k_j}}=\\
= \frac{1}{n_{k_j}}\left[\overbrace{\cco C_1+\dots + \cco C_1}^{p^1_{n_{k_j}} -\text{times}} + \overbrace{\cco C_2+\dots +\cco C_2}^{p^2_{n_{k_j}} -\text{times}} + \dots + \overbrace{\cco C_d+\dots +\cco C_d}^{p^d_{n_{k_j}} -\text{times}}\right] \\
=\frac{p^1_{n_{k_j}}}{n_{k_j}} \cco C_1+\frac{p^2_{n_{k_j}}}{n_{k_j}}\cco C_2 +\dots +\frac{p^d_{n_{k_j}}}{n_{k_j}}\cco C_d, 
\end{gather*}
from which \eqref{eq4.9} follows at once, by letting $j\to+\infty$. Thus Claim 1 is valid.

{\bf Claim 2.} We have $h(C,\Gamma)\leq \ep$.

Indeed, from \eqref{eq4.5} and \eqref{eq4.9} it follows that
\begin{eqnarray}
\frac{\overline{\cco X_1+\dots +\cco X_{n_{k_j}}}}{n_{k_j}}\xrightarrow{F} C \qquad \text{as}\; j\to+\infty \label{eq4.13}\\
\frac{\overline{\cco C^{\prime}_1+\dots +\cco C^{\prime}_{n_{k_j}}}}{n_{k_j}}\xrightarrow{F} \Gamma \qquad \text{as}\; j\to+\infty. \label{eq4.14}
\end{eqnarray}
Furthermore, by construction, the sequence $\{C^{\prime}_n\}_{n=1}^{+\infty}$ satisfies $h(X_n,C^{\prime}_n)\leq\ep$ for every $n\in\bN$. Hence
\begin{equation}\label{eq4.15}
h\left( \frac{X_1+\dots+X_{n_{k_j}}}{n_{k_j}}, \frac{C^{\prime}_1+\dots+C^{\prime}_{n_{k_j}}}{n_{k_j}}\right) \leq \frac{1}{n_{k_j}}\sum_{i=1}^{n_{k_j}} h(X_i,C^{\prime}_i)<\ep\quad \text{for every}\;j\in\bN
\end{equation}
and thus, a fortiori,
\begin{equation*}
h\left( \frac{\overline{\cco X_1+\dots +\cco X_{n_{k_j}}}}{n_{k_j}},\frac{\overline{\cco C^{\prime}_1+\dots +\cco C^{\prime}_{n_{k_j}}}}{n_{k_j}}\right) <\ep \quad \text{for every}\;j\in\bN.
\end{equation*}
Then by virtue of Proposition \ref{proposition3.4}, in view of \eqref{eq4.13}-\eqref{eq4.15}, it follows that $h(C,\Gamma)\leq\ep$, and thus Claim 2 is proved.

With the help of Claims 1 and 2 we are now in a position to complete the proof of $(jj)$. Since $z\in C$ and, by Claim 2, $h(C,\Gamma)\leq\ep$, there exists $y\in\Gamma$ such that
\begin{equation}\label{eq4.16}
\| z-y\|<2\ep.
\end{equation}
As $y\in\Gamma$, then by \eqref{eq4.8} there is a $j_0\in\bN$ such that
\begin{equation}\label{eq4.17}
d\left( y, \frac{C^{\prime}_1+\dots+C^{\prime}_{n_{k_j}}}{n_{k_j}}\right) <\ep \qquad \text{for every}\; j\geq j_0.
\end{equation}
Taking into account \eqref{eq4.16}, \eqref{eq4.17} and \eqref{eq4.15}, for every $j\geq j_0$ we have
\begin{gather*}
d\left( z, \frac{X_1+\dots+X_{n_{k_j}}}{n_{k_j}}\right)\leq \|z-y\| +\\
+d\left( y, \frac{C^{\prime}_1+\dots+C^{\prime}_{n_{k_j}}}{n_{k_j}}\right) +h\left(\frac{C^{\prime}_1+\dots+C^{\prime}_{n_{k_j}}}{n_{k_j}}, \frac{X_1+\dots+X_{n_{k_j}}}{n_{k_j}}\right) <4\ep,
\end{gather*}
a contradiction to \eqref{eq4.6}, since $\ep<\theta/8$. Hence also $(jj)$ is valid and the proof of the theorem is complete.
\end{proof}

\begin{theorem}\label{theo4.2}
Let $\cZ\subset \fC(\bE)$ be a nonempty compact set in the Pompeiu-Hausdorff metric $h$ of $\fC(\bE)$. Let $\{X_n\}_{n=1}^{+\infty} \subset\cZ$, and let $C\in \fC_c(\bE)$ be a compact and convex set. Then
\begin{equation}\label{eq4.18}
\frac{\overline{\cco X_1+\dots +\cco X_n}}{n}\xrightarrow{F} C \qquad \text{as}\; n\to+\infty
\end{equation}
implies
\begin{equation*}
\frac{\overline{X_1+\dots +X_n}}{n}\xrightarrow{h} C \qquad \text{as}\; n\to+\infty.
\end{equation*}
\end{theorem}

\begin{proof}
Let us associate to the compact set $\cZ\subset\fC(\bE)$ the following subsets $\cU$ and $\text{co}\cU$ of $\fC_c(\bE)$ defined as follows,
\begin{gather*}
\cU = \{Y\in\fC_c(\bE)\;|\; Y=\cco X \;\text{for some}\;X\in\cZ\}\\
\text{co}\cU = \{Y\in\fC_c(\bE)\;|\; Y=\overline{\sum_{i=1}^d \l_i Y_i},\;\text{where}\; Y_i\in\cU, \l_i\in [0,1],\sum_{i=1}^d\l_i=1, d\in\bN\}
\end{gather*}
Since $\cZ$ is compact and, for arbitrary $A_1,A_2\in\fC(\bE)$, $h(\cco A_1,\cco A_2)\leq h(A_1,A_2)$, it follows that $\cU$ is compact. Hence, by Proposition \ref{proposition3.6}, the set $\cZ_0 =\text{cl}[\text{co}\cU]$ (the closure is taken in the $h$-metric) is a convex and compact set in $\fC_c(\bE)$.

Let $\{X_n\}_{n=1}^{+\infty}\subset\cZ$ be a sequence satisfying \eqref{eq4.18} with $C\in \fC_c(\bE)$ compact and convex. Hence, by Theorem \ref{theo4.1},
\begin{equation}\label{eq4.19}
\frac{\overline{X_1+\dots +X_n}}{n}\xrightarrow{F} C \qquad \text{as}\; n\to+\infty.
\end{equation}

Let $\ep>0$ be arbitrary. From \eqref{eq4.19}, for each $u\in C$ there exists $n_u\in\bN$ such that
\begin{equation}\label{eq4.20}
\frac{\overline{X_1+\dots +X_n}}{n} \cap B(u,\ep) \neq \emptyset\qquad \text{for every}\; n\geq n_u.
\end{equation}
As $\{B(u,\ep)\}_{u\in C}$ is an open covering of $C$ it admits a finite subcovering, $i.e.$, there exist $u_1,\dots,u_d\in C$ such that
\begin{equation}\label{eq4.21}
C\subset \bigcup_{i=1}^{d} B\left( u_i,\frac{\ep}{2}\right) .
\end{equation}
Set $n^{\prime}=\max \{n_{u_1},\dots,n_{u_d}\}$. Let $n\geq n^{\prime}$ be arbitrary. By \eqref{eq4.20} (with $u_i$ in the place of $u$) it follows that each ball $B(u_i,\ep/2)$, $i=1,\dots,d$, contains some points of $(\overline{X_1+\dots +X_n})/n$ and thus, by \eqref{eq4.21},
\begin{equation}\label{eq4.22}
C\subset \frac{\overline{X_1+\dots +X_n}}{n} +\ep B \qquad \text{for every}\; n\geq n^{\prime}.
\end{equation}
On the other hand for each $n\in\bN$
\begin{equation}\label{eq4.23}
\cco \frac{\overline{X_1+\dots +X_n}}{n} = \frac{\overline{\cco X_1+\dots +\cco X_n}}{n}
\end{equation}
and thus, as each $\cco X_i$, $i=1,\dots,n$, is in the convex set $\cZ_0$, it follows that $\cco (\,\overline{X_1+\dots +X_n}\,)/n \in\cZ_0$ for every $n\in\bN$. Moreover, $\cZ_0$ and $C$ are compact and 
\begin{equation*}
\cco \frac{X_1+\dots +X_n}{n}\xrightarrow{F} C\qquad \text{as}\; n\to+\infty, 
\end{equation*}
by virtue of \eqref{eq4.23} and \eqref{eq4.18}. Hence by Proposition \ref{proposition3.5},
\begin{equation*}
\cco \frac{\overline{X_1+\dots +X_n}}{n} \xrightarrow{h} C\qquad \text{as}\;n\to+\infty.
\end{equation*}
Consequently, there exists $n^{\prime\prime}\in\bN$ such that
\begin{equation}\label{eq4.24}
\frac{\overline{X_1+\dots +X_n}}{n} \subset \cco \frac{\overline{X_1+\dots +X_n}}{n} \subset C+\ep B \qquad \text{for every}\; n\geq n^{\prime\prime}.
\end{equation}
Combining \eqref{eq4.22} and \eqref{eq4.24} yields
\begin{equation*}
h\left( \frac{\overline{X_1+\dots +X_n}}{n} , C\right) <\ep \qquad \text{for every}\; n\geq n_0,
\end{equation*}
where $n_0 =\max\{n^{\prime},n^{\prime\prime}\}$. Hence $(\,\overline{X_1+\dots +X_n}\,)/n \xrightarrow{h} C$ as $n\to+\infty$, completing the proof.
\end{proof}

\section{Expectations of random sets in non separable Banach spaces}\label{section5}

In this section we define a notion of expectation $E(F)$ for random sets $F: \Omega \to \fC(\bE)$ where the underlying Banach space $\bE$ is not necessarily separable. A few properties are reviewed and some proofs will be given since, in our non separable setting, the Kuratowski-Ryll Nardzewski theorem is not valid. For convex valued random sets $F: \Omega \to \fC_c(\bE)$ the expectation is proven to be consistent with that obtained by using the classic R\text{\r{a}}dstr$\Ddot{\text o}$m embedding.

In what follows $(\O,\sA,P)$ is a complete probability space without atoms. Following Hille and Phillips \cite[pp. 71-73]{Hille1957}, a map $G:\O\to\fC(\bE)$ is said to be \textit{countably-valued} if it admits a \textit{representation} of the form
\begin{equation*}
G(\o)=\sum_{i=0}^{+\infty} A_i \chi_{\O_i}(\o) \qquad \o\in\O
\end{equation*}
where $\{A_i\}_{i=1}^{+\infty}\subset \fC(\bE)$ and $\{\O_i\}_{i=1}^{+\infty}$ is a measurable partition of $\O$, $i.e.$, $\{\O_i\}_{i=1}^{+\infty}$ is a partition of $\O$ consisting of sets $\O_i\in\sA$.

\begin{Remark}
It is worth noting that a countably valued map can have several representations.
\end{Remark}

A map $F: \Omega \to \fC(\bE)$ is said to be \textit{strongly measurable} if there exists a sequence $\{G_n\}_{n=1}^{+\infty}$ of countably-valued maps $G_n : \Omega \to \fC(\bE)$ which converge to $F$ uniformly almost surely (a.s.) on $\O$, $i.e.$,
\begin{equation*}
\lim_{n\to+\infty} h_{\infty}(G_n,F)=0 \qquad \text{where}\quad h_{\infty}(G_n,F)=\underset{\o\in\O}{\text{ess sup}}\, h(G_n(\o),F(\o)).
\end{equation*}
It is evident that a countably-valued map is strongly measurable.

A single-valued map $f:\O\to\bE$ is said to be \textit{measurable} if there exists a sequence of countably-valued maps $g_n:\O\to\bE$ which converges uniformly a.s. on $\O$ to $f$, $i.e.$,
\begin{equation*}
\lim_{n\to+\infty} \| g_n-f\|_{\infty}=0  \qquad \text{where}\quad \| g_n-f\|_{\infty}=\underset{\o\in\O}{\text{ess sup}}\, \|g_n(\o)-f(\o)\|.
\end{equation*}

A map $F: \Omega \to \fC(\bE)$ is said to be \textit{weakly measurable} if for every $x\in\bE$ the map $\o\to d(x,F(\o))$ is measurable.

The definition of a \textit{simple} map $G: \Omega \to \fC(\bE)$ is as in the single-valued case.

\begin{Remark}
Let $F: \Omega \to \fC(\bE)$ be a strongly measurable map. Then, we have:\\
$(a_1)$ $F$ is weakly measurable; $(a_2)$ $\|F\|$ is measurable, where $\|F\| :\O\to [0,+\infty)$ is defined by $\|F\|(\o)=\|F(\o)\|$, $\o\in\O$; $(a_3)$ $\cco F$ is strongly measurable, where $\cco F:\O\to\fffC$ is defined by $(\cco F)(\o)=\cco F(\o)$, $\o\in\O$. Moreover, $F: \Omega \to \fC(\bE)$ is strongly measurable if and only if there exists a sequence $\{G_n\}_{n=1}^{+\infty}$ of simple maps $G_n: \Omega \to \fC(\bE)$ converging to $F$ a.s. on $\O$.
\end{Remark}

The meaning of $L^p(\O,\bE)$, $1\leq p\leq +\infty$, with the usual $\|\cdot\|_p$-norm is the standard one. In particular,
\begin{equation*}
L^1(\O,\bE) = \{ f: \O\to\bE \;|\; f\; \text{is measurable and}\; \int_{\O}\negthickspace\|f\|dP <+\infty\}
\end{equation*}
and thus each $f\in L^1(\O,\bE)$ is integrable. Here and in what follows integrability and measurability for a function $f$ are understood in the sense of Bochner.

Now set
\begin{gather*}
\bL^1 (\O,\fC(\bE)) = \left\lbrace F: \O\to\fC(\bE) \;|\; F\; \text{is s-measurable and}\; \int_{\O}\negthickspace\|F\|dP <+\infty\right\rbrace \\
\bL^1 (\O,\fC_c(\bE)) = \left\lbrace F: \O\to\fC_c(\bE) \;|\; F\; \text{is s-measurable and}\; \int_{\O}\negthickspace\|F\|dP <+\infty\right\rbrace
\end{gather*}
and, furthermore, let
\begin{gather*}
\bZ^1 (\O,\fC(\bE)) = \left\lbrace F: \O\to\fC(\bE) \;|\; F\; \text{is count-valued and}\; \int_{\O}\negthickspace\|F\|dP <+\infty\right\rbrace \\
\bZ^1(\O,\fC_c(\bE)) = \left\lbrace F: \O\to\fC_c(\bE) \;|\; F\; \text{is count-valued and}\; \int_{\O}\negthickspace\|F\|dP <+\infty\right\rbrace.
\end{gather*}
In the above definitions s-measurable and count-valued stand for strongly measurable and countable-valued respectively. In what follows an element $F\in\bL^1 (\O,\fC(\bE))$ will also be called a \textit{random} set.\\
For any $F\in \bL^1 (\O,\fC(\bE))$ we define
\begin{gather*}
\text{S}^1_F = \left\lbrace f: \O\to\bE \;|\; f\; \text{is a measurable selector of}\; F\right\rbrace \\
\text{V}^1_F = \left\lbrace f: \O\to\bE \;|\; f\; \text{is a countably-valued selector of}\; F\right\rbrace.
\end{gather*}
Here by a \textit{measurable} (resp. \textit{countably-valued}) selection of $F$ we mean a measurable (resp. countably-valued) map $f:\O\to\bE$ satisfying $f(\o)\in F(\o)$, $\o\in\O$. Clearly $\text{S}^1_F$ and $\text{V}^1_F$ are (perhaps empty) subsets of $L^1(\O,\bE)$. The meaning of $\bL^1 (\O,\cW)$ and $\bZ^1 (\O,\cW)$, where $\cW$ is a nonempty subset of $\fB(\bE)$, is evident.

\begin{Remark}\label{remark5.3}
$(a_1)$ For each $G\in \bZ^1 (\O,\fC(\bE))$ the set $\text{V}^1_G$ is nonempty; $(a_2)$ Let $G\in \bZ^1 (\O,\fC(\bE))$ and $g:\O\to\bE$ be bounded countably-valued maps satisfying, for some $\a>0$, $d(g(\o),G(\o))<\a$ a.s. on $\O$. Then there exists a bounded countably-valued map $\tilde{g}\in\text{V}^1_G$ such that $\|\tilde{g}-g\|_{\infty}\leq\a$.
\end{Remark}

The following measurable selection theorem is a variant of the Kuratowski-Ryll Nardzewski theorem \cite{Kuratowski1965, kuratowski1966}.

\begin{theorem}\label{theo5.1}
For each $F\in \bL^1 (\O,\fC(\bE))$ the set $\text{S}^1_F$ is nonempty.
\end{theorem}
\begin{proof}
Let $F\in \bL^1 (\O,\fC(\bE))$. Suppose first that $F$ is bounded. Let $\{G_n\}_{n=1}^{+\infty}\subset \bZ^1 (\O,\fC(\bE))$ be a sequence of countably-valued maps $G_n:\O\to\fC(\bE)$, given by
\begin{equation}\label{eq5.1}
G_n(\o) = \sum_{i=1}^{+\infty}A_i^{n} \chi_{\O_i^{n}}(\o) \qquad \o\in\O
\end{equation}
where, for each fixed $n\in\bN$, $\{A_i^{n}\}_{i=1}^{+\infty}\subset\ffC$ and $\{\O_i^{n}\}_{i=1}^{+\infty}$ is a measurable partition of $\O$, such that $G_n$ converges to $F$ uniformly a.s. on $\O$. Without loss of generality, passing to a subsequence if necessary, with the same notation as before, we can assume that
\begin{equation*}
 h_{\infty}(G_n,F)<\frac{\ep_n}{2} \qquad \text{for every}\; n\in\bN,
\end{equation*}
where $\ep_n =1/2^n$. Hence,
\begin{equation*}
 h_{\infty}(G_n,G_{n+1})<\ep_n \qquad \text{for every}\; n\in\bN.
\end{equation*}
Fix $g_1\in\text{V}^1_{G_1}$. Since $d(g_1(\o),G_2(\o))\leq h(G_1(\o),G_2(\o))\leq h_{\infty}(G_1,G_2)<\ep_1$\, a.s. on $\O$, in view of Remark \ref{remark5.3}, $(a_1)$, there exists $g_2\in\text{V}^1_{G_1}$ such that
\begin{equation*}
 \| g_2-g_1\|_{\infty}\leq\ep_1.
\end{equation*}
Clearly $g_2\in L^{\infty}(\O,\bE)$, for $g_2$ is bounded. Similarly, we have $d(g_2(\o),G_3(\o))\leq h(G_2(\o),G_3(\o))\leq h_{\infty}(G_2,G_3)<\ep_2$ a.s. on $\O$, and thus there exists $g_3\in\text{V}^1_{G_3}$ satisfying 
\begin{equation*}
 \| g_3-g_2\|_{\infty}\leq\ep_2,
\end{equation*}
and evidently $g_3\in L^{\infty}(\O,\bE)$. By this procedure one can construct a sequence $\{g_n\}_{n=1}^{+\infty}\subset L^{\infty}(\O,\bE)$ with $g_n\in\text{V}^1_{G_n}$ such that
\begin{equation*}
 \| g_{n+1}-g_n\|_{\infty}\leq\ep_n \qquad \text{for every}\; n\in\bN.
\end{equation*}
Now $\{g_n\}_{n=1}^{+\infty}$ is a Cauchy sequence in $L^{\infty}(\O,\bE)$ and thus it converges to some $\phi\in L^{\infty}(\O,\bE)$. Since
\begin{equation*}
d(\phi(\o),F(\o))\leq \|\phi(\o)-g_n(\o)\|+ d(g_n(\o),G_n(\o))+ h(G_n(\o),F(\o)) \quad \text{a.s. on}\;\O
\end{equation*}
and, for $n\to+\infty$, $\|\phi(\o)-g_n(\o)\|_{\infty}\to 0$ and $h_{\infty}(G_n(\o),F(\o))\to 0$ as $n\to+\infty$, it follows that $\phi$ is a measurable selection of $F$ and thus $\phi\in \text{S}^1_F$.

Now suppose that $F$ is not bounded. For $k\in\bN$ let $\O_k=\{\o\in\O\;|\; k-1\leq \|F(\o)\|< k\}$ and observe that $\{\O_k\}_{k=1}^{+\infty}$ is a measurable partition of $\O$. For $k\in\bN$ define $F_k:\O\to\ffC$ by $F_k(\o)=F(\o)$, $\o\in\O_k$. As $F_k$ is strongly measurable and bounded on $\O_k$ it admits a measurable selection, say $f_k:\O_k\to\bE$. Then the map $f:\O\to\bE$ given by
\begin{equation*}
f(\o) = \sum_{i=1}^{+\infty}f_k(\o) \chi_{\O_k}(\o) \qquad \o\in\O,
\end{equation*}
is a measurable selection of $F$. Therefore $\text{S}^1_F\neq\emptyset$, completing the proof.
\end{proof}

\begin{proposition}\label{proposition5.2}
Let $F\in\bL^1 (\O,\fC(\bE))$ and $r>0$. Then, for each $\s>0$,
\begin{equation}\label{eq5.2}
\text{S}^1_{\overline{F+ rB}}\subset \text{S}^1_F + (r+\s)\text{S}^1_B.
\end{equation}
\end{proposition}
\begin{proof}
By Theorem \ref{theo5.1} the sets $\text{S}^1_{\overline{F+ rB}}$ and $\text{S}^1_F$ are nonempty and, clearly, $\text{S}^1_B\neq\emptyset$.

{\bf Claim.} Let $\s>0$. For every $\phi\in\text{S}^1_{\overline{F+ rB}}$ there exist $f\in\text{S}^1_F$ and $\g\in\text{S}^1_{B}$ such that
\begin{equation}\label{eq5.3}
\phi(\o)= f(\o) +(r+\sigma)\g(\o) \qquad \text{a.s. on}\;\O.
\end{equation}
Indeed, suppose that $F$ and $\phi$ are bounded, and let $\phi\in\text{S}^1_{\overline{F+ rB}}$ be arbitrary. Thus
\begin{equation}\label{eq5.4}
\phi(\o)\in \overline{F(\o) +rB} \qquad \text{a.s. on}\;\O.
\end{equation}
Let $\phi_0:\O\to \bE$ be a countably-valued map satisfying
\begin{equation}\label{eq5.5}
\|\phi_0-\phi\|_{\infty}<\frac{\s}{8},
\end{equation}
and let $\{G_n\}_{n=1}^{+\infty}\subset \bZ^1(\O,\ffC)$ be a sequence of countably-valued maps $G_n:\O\to \ffC$, given by \eqref{eq5.1}, converging to $F$ uniformly a.s. on $\O$. Passing to a subsequence if necessary, we can assume without loss of generality that
\begin{equation}\label{eq5.6}
 h_{\infty}(G_n,F)<\frac{\ep_n}{2} \qquad \text{for every}\; n\in\bN,
\end{equation}
where $\ep_n=\s/2^{n+2}$. Hence
\begin{equation}\label{eq5.7}
 h_{\infty}(G_n,G_{n+1})<\ep_n \qquad \text{for every}\; n\in\bN.
\end{equation}
By virtue of \eqref{eq5.5}, \eqref{eq5.4} and \eqref{eq5.6} (with $n=1$) we have
\begin{gather*}
\phi_0(\o)\in \phi(\o) +\frac{\sigma}{8}B\subset \overline{F(\o)+rB}+\frac{\sigma}{4}B\subset (F(\o)+rB+\ep B)+\frac{\sigma}{4}B\\
\subset G_1(\o)+\left( r+\frac{\s}{4}+ \ep_1\right) B \subset G_1(\o)+\left( r+\frac{\s}{2}\right) B \qquad \text{a.s. on}\;\O,
\end{gather*}
and thus $d(\phi_0(\o),G_1(\o))<r+\s/2$ a.s. on $\O$. Then by virtue of Remark \ref{remark5.3} $(a_2)$, there exists a $g_1\in\text{V}^1_{G_1}$ such that
\begin{equation}\label{eq5.8}
\|g_1-\phi_0\|_{\infty}\leq r+\frac{\s}{2}B.
\end{equation}
Now, by \eqref{eq5.7} we have $d(g_1(\o),G_2(\o))\leq h(G_1(\o),G_2(\o))\leq h_{\infty}(G_1,G_2)<\ep_1$\, a.s. on $\O$, and thus there exists $g_2\in\text{V}^1_{G_2}$ such that $\|g_2-g_1\|_{\infty}\leq \ep_1$. By this procedure one can construct a sequence $\{g_n\}_{n=1}^{+\infty}\subset L^{\infty}(\O,\bE)$, with $g_n\in\text{V}^1_{G_n}$, satisfying
\begin{equation}\label{eq5.9}
\|g_{n+1}-g_n\|_{\infty}\leq \ep_n \qquad \text{for every}\; n\in\bN.
\end{equation}
Evidently, \eqref{eq5.9} implies that $\{g_n\}_{n=1}^{+\infty}$ is a Cauchy sequence in $L^{\infty}(\O,\bE)$ and thus, for some $f\in L^{\infty}(\O,\bE)$, $\|g_n-f\|_{\infty}\to 0$ as $n\to+\infty$. Moreover, by \eqref{eq5.6} we have $h_{\infty}(G_n,F)\to 0$ as $n\to+\infty$. Since $g_n(\o)\in G_n(\o)$ a.s. on $\O$, it follows that $f(\o)\in F(\o)$ a.s. on $\O$, that is $f$ is a measurable selection of $F$, and hence $f\in \text{S}^1_{F}$.

On the other hand in view of \eqref{eq5.8} and \eqref{eq5.9},
\begin{gather*}
\|g_{n+1}(\o)-\phi_0(\o)\| \leq \|g_1(\o)-\phi_0(\o)\| + \sum_{k=1}^n \|g_{k+1}(\o)-g_k(\o)\|\\
\leq r+\frac{\s}{2}+ \sum_{k=1}^n\ep_k < r+\frac{3}{4}\s \qquad \text{a.s. on}\;\O,
\end{gather*}
from which letting $n\to+\infty$ one has $\|f-\phi_0\|_{\infty}\leq r+3\s/4$. The latter and \eqref{eq5.5} imply
\begin{equation}\label{eq5.10}
\|\phi-f\|_{\infty}< r+\s.
\end{equation}
Now define $\g:\O\to\bE$ by $\g(\o)=(\phi(\o)-f(\o))/(r+\s)$ for $\o\in\O$ a.s. Since $\g\in\text{S}^1_{B}$, by \eqref{eq5.10}, $f\in\text{S}^1_{F}$ and, moreover,
\begin{equation*}
\phi(\o)=f(\o)+(\phi(\o)-f(\o))=f(\o)+(r+\s)\g(\o) \qquad \text{a.s. on}\;\O,
\end{equation*}
it follows that \eqref{eq5.3} is valid, whenever $F$ and $\phi$ are bounded. The general case, when $F$ and $\phi$ are not necessarily bounded, can be treated as in Theorem \ref{theo5.1} and thus the proof is omitted. Hence the Claim is proved. The statement \eqref{eq5.2} is an immediate consequence of the Claim. This completes the proof.
\end{proof}

Let us recall that the sum of a series $\sum_{i=1}^{+\infty}C_i$, where $C_i\in\fffC$, is a set $C\in\fffC$ (if it exists) such that $\lim_{n\to+\infty}h(\overline{\sum_{i=1}^{n}C_i},C)=0$. Moreover, the sum $C$ exists and is unique if the series is absolutely convergent, $i.e.$, if \,$\sum_{i=1}^{+\infty}\|C_i\|<+\infty$.

\begin{definition}\label{def5.1}
Let $F\in\bL^1(\O,\ffC)$. The \textit{Aumann integral} of $F$ on $\O$ (see \cite{Aumann1965, Hu1997}) is defined by
\begin{equation*}
\aint FdP = \overline{\left\lbrace\int_{\O}fdP \;|\; f\in\text{S}^1_{F}\right\rbrace}.
\end{equation*}
\end{definition}

\begin{Remark}\label{remark5.4}
The above definition is meaningful since $\text{S}^1_{F}\neq \emptyset$, by Theorem \ref{theo5.1}. Moreover, the Aumann integral of $F$ is a set $C\in\fffC$ which, obviously, exists and is unique.
\end{Remark}

Let $G\in\bZ^1(\O,\fffC)$ be a countably-valued map $G:\O\to\fffC$ with representation given by
\begin{equation*}
G(\o)=\sum_{i=1}^{+\infty} A_i\chi_{\O_i}(\o) \qquad\o\in\O,
\end{equation*}
where $\{A_i\}_{i=1}^{+\infty}\subset\fffC$, and $\{\O_i\}_{i=1}^{+\infty}$ is a measurable partition of $\O$. The \textit{Hukuhara integral} of $G$ on $\O$ is defined by
\begin{equation*}
\hint GdP = \sum_{i=1}^{+\infty} A_i P(\O_i).
\end{equation*}

\begin{Remark}\label{remark5.5}
Since the above series is absolutely convergent, it converges to a set $C\in\fffC$, which exists and is unique. Moreover this set $C$ is independent of the representation of $G$ and thus the Hukuhara integral of $G\in\bZ^1(\O,\fffC)$ is meaningful.
\end{Remark}

\begin{definition}\label{def5.2}
Let $F\in\bL^1(\O,\fffC)$ and let $\{G_n\}_{n=1}^{+\infty}\subset \bZ^1(\O,\fffC)$ be a sequence of countably-valued maps converging to $F$ uniformly a.s. on $\O$. The \textit{Hukuhara integral} of $F$ on $\O$ \cite{Hukuhara1967} is defined by
\begin{equation*}
\hint FdP = \lim_{n\to+\infty}\hint G_ndP.
\end{equation*}
\end{definition}

\begin{Remark}\label{remark5.6}
As in \cite[p.79]{Hille1957}, it can be shown that the above limit exists and is unique, actually it is a set $C\in\fffC$. Moreover this set $C$ is independent of the particular sequence $\{G_n\}_{n=1}^{+\infty}\subset \bZ^1(\O,\fffC)$ used in the definition. Therefore the Hukuhara integral of $F\in\bL^1(\O,\fffC)$ is meaningful.
\end{Remark}

It is worth noting that the Aumann and Hukuhara integrals are well defined also when $\O$ is replaced by a set $\O^{\prime}\in\sA$.

\begin{proposition}\label{proposition5.3}
Let $\O^{\prime}\in\sA$. Then for each $C\in\fffC$,
\begin{equation}\label{eq5.11}
\Aint CdP = \Hint CdP.
\end{equation}
Moreover for each $C\in\ffC$,
\begin{equation}\label{eq5.12}
\Aint CdP = \Hint (\cco C)dP.
\end{equation}
\end{proposition}
\begin{proof}
Indeed for each $\s\in\text{S}^1_C$,
\begin{equation*}
\int_{\O^{\prime}}\negthickspace\s dP \in CP(\O^{\prime})=\Hint CdP,
\end{equation*}
which implies that the set on the left hand side of \eqref{eq5.11} is contained in the set on the right hand side. Moreover,
\begin{equation*}
\Hint C dP = CP(\O^{\prime})=\left\lbrace\int_{\O^{\prime}}\negthickspace cdP \;|\;c\in C\right\rbrace \subset \Aint CdP,
\end{equation*}
and thus \eqref{eq5.11} holds.

Now consider \eqref{eq5.12}. Clearly by virtue of \eqref{eq5.11}
\begin{equation*}
\Aint C dP \subset \Aint(\cco C)dP=\Hint(\cco C)dP.
\end{equation*}

To prove the reverse inclusion, let $\ep>0$ and consider an arbitrary point $\bar{\xi}\in(\cco C)P(\O^{\prime})$. Take a $\xi\in(\co C)P(\O^{\prime})$ such that $\|\xi-\bar{\xi}\|<\ep$. Clearly, for some $c_i\in C$ and $\l_i>0$, $i=1,\dots,n$, with $\l_1+\dots+\l_n=1$, we have $\xi=P(\O^{\prime})(\l_1 c_1+\dots+\l_n c_n)$. Now $(\O,\sA,P)$ is a complete probability space without atoms, and thus by Liapunoff convexity theorem \cite{Diestel1967} there exists a measurable partition of $\O^{\prime}$, say $\{\O_i\}_{i=1}^{n}$, such that $P(\O_i)=\l_i P(\O^{\prime})$, $i=1,\dots,n$. Define $\s: \O^{\prime}\to\bE$ by $\s(\o)=c_1\chi_{\O_1}(\o)+\dots+c_n\chi_{\O_n}(\o)$, for $\o\in\O^{\prime}$. Clearly $\s\in\text{S}^1_C$. Moreover
\begin{equation*}
\xi=\sum_{i=1}^n \l_i c_i P(\O^{\prime})= \sum_{i=1}^n c_i P(\O_i)=\sum_{i=1}^n \int_{\O^{\prime}}c_i\chi_{\O_i}dP = \int_{\O^{\prime}}\s dP \in\Aint CdP,
\end{equation*}
and hence, as $\bar{\xi}\in \xi+\ep B$,
\begin{equation*}
\bar{\xi} \in \Aint CdP +\ep B.
\end{equation*}
Since $\bar{\xi}\in(\cco C)P(\O^{\prime})$ and $\ep>0$ are arbitrary, it follows that
\begin{equation*}
\Hint(\cco C)dP \subset \Aint CdP,
\end{equation*}
and thus \eqref{eq5.12} is valid. This completes the proof.
\end{proof}

By virtue of Proposition \ref{proposition5.3} we have

\begin{proposition}\label{proposition5.4}
Let $G_1\in\bZ^1(\O,\fffC)$, $G_2\in\bZ^1(\O,\ffC)$ and let $\O^{\prime}\in\sA$. Then,
\begin{equation*}
\Aint G_1 dP=\Hint G_1 dP \qquad\qquad \Aint G_2 dP=\Hint(\cco G_2)dP.
\end{equation*}
\end{proposition}

\begin{proposition}\label{proposition5.5}
Let $F\in\bL^1(\O,\fffC)$ and let $\O^{\prime}\in\sA$. Then,
\begin{equation}\label{eq5.13}
\Aint F dP=\Hint F dP.
\end{equation}
\end{proposition}
\begin{proof}
Suppose first that $F\in\bL^1(\O,\fffC)$ is bounded. Let $\ep>0$ and let $\bar{\xi}$ be an arbitrary point in the Aumann integral of $F$. Then for some $f\in\text{S}^1_F$, setting $\xi=\int_{\O^{\prime}}fdP$, we have
\begin{equation}\label{eq5.14}
\|\bar{\xi}-\xi\|<\ep.
\end{equation}
Let $\{\phi_n\}_{n=1}^{+\infty}$ and $\{G_n\}_{n=1}^{+\infty}$ be sequences of countably-valued maps converging, respectively, to $f$ and $F$, uniformly a.s. on $\O$. Without loss of generality we can assume that for each $n\in\bN$, $\phi_n$ and $G_n$ have representations given by
\begin{equation*}
\phi_n(\o) = \sum_{i=1}^{+\infty}a_i^n \chi_{\O_i^n}(\o) \qquad G_n(\o) = \sum_{i=1}^{+\infty}A_i^n \chi_{\O_i^n}(\o), \qquad \o\in\O^{\prime},
\end{equation*}
where $\{a_i^n\}_{i=1}^{+\infty}\subset\bE$, $\{A_i^n\}_{i=1}^{+\infty}\subset\fffC$, $\{\O_i^n\}_{i=1}^{+\infty}$ is a measurable partition of $\O^{\prime}$ and, moreover,
\begin{equation}\label{eq5.15}
\|\phi_n-f\|_{\infty}<\frac{1}{n} \qquad h_{\infty}(G_n,F)<\frac{1}{n}, \qquad n\in\bN.
\end{equation}
Now, for each $n\in\bN$,
\begin{equation*}
\phi_n(\o)\in f(\o)+\|\phi_n-f\|_{\infty}B\subset F(\o)+\frac{1}{n}B \subset G_n(\o)+\frac{2}{n}B \qquad \text{a.s. on}\;\O^{\prime}
\end{equation*}
and thus, denoting by $\tilde{G}_n:\O^{\prime}\to\fffC$ a map with representation given by
\begin{equation*}
\tilde{G}_n(\o) = \sum_{i=1}^{+\infty}\tilde{A}_i^n \chi_{\O_i^n}(\o),\qquad\qquad \o\in\O^{\prime}
\end{equation*}
where $\tilde{A}_i^n=\overline{A_i^n+B/n}$\,, it follows that $\phi_n(\o)\in\tilde{G}_n(\o)$ a.s. on $\O^{\prime}$. Consequently
\begin{equation}\label{eq5.16}
\int_{\O^{\prime}} \phi_n dP\in\Hint \tilde{G}_n dP.
\end{equation}
Evidently $\{\tilde{G}_n\}_{n=1}^{+\infty}\subset \bL^1(\O^{\prime},\fffC)$ is a sequence of countably-valued maps $\tilde{G}_n:\O^{\prime}\to\fffC$ converging to $F$ uniformly a.s. on $\O^{\prime}$. Then letting $n\to+\infty$, from \eqref{eq5.16} in view of Remark \ref{remark5.6} one has
\begin{equation*}
\xi=\int_{\O^{\prime}} f dP\in\Hint F dP.
\end{equation*}
Since $\tilde{\xi}$ is an arbitrary point in the Aumann integral of $F$, from the latter and \eqref{eq5.14} it follows that
\begin{equation}\label{eq5.17}
\Aint F dP\subset \Hint F dP +\ep B.
\end{equation}

Let us prove the reverse inclusion, obtained by interchanging the roles of the Aumann and Hukuhara integrals. To this end, fix $n_0\in\bN$ so that $n_0>4/\ep$. Then for every $n\geq n_0$ by virtue of Proposition \ref{proposition5.4} and \eqref{eq5.15} we have
\begin{gather}\label{eq5.18}
\Hint G_n\, dP=\Aint G_n dP\subset \Aint\; \left(\,\overline{F+\frac{1}{n}B}\,\right) dP\nonumber\\
\subset \Aint\;\left(\,\overline{ F+\frac{\ep}{4}B}\,\right)\,dP \subset \Aint F\,dP +\frac{3\ep}{4} B,
\end{gather}
where the last inclusion holds because, by Proposition \ref{proposition5.2} (with $r, \s$ and $\O$ replaced by $\ep/r, \ep/4$ and $\O^\prime$ respectively),
\begin{equation*}
\text{S}^1_{\overline{F+\frac{\ep}{4}B}} \subset  \text{S}^1_F+\frac{3\ep}{4} \text{S}^1_B.
\end{equation*}
From (\ref{eq5.18}.18) letting $n\to+\infty$ one has
\begin{equation*}
\Hint F\,dP \subset \Aint F\,dP +\ep B.
\end{equation*}
Since the Aumann and the Hukuhara integrals are closed sets and $\ep>0$ is arbitrary, from the latter and \eqref{eq5.17} it follows that \eqref{eq5.13} is valid. Thus the statement is proved, whenever $F\in\bL^1(\O,\fffC)$ is bounded.

Suppose that $F\in\bL^1(\O,\fffC)$ is not bounded. For $n\in\bN$ set $\O_n=\{\o\in\O^{\prime}\;|\; \|F(\o)\|>n\}$ and define
\[ F_n(\o)=\begin{cases}
            F(\o)\qquad &\text{if}\;\o\in\O^{\prime}\setminus\O_n,\\
	    \;\;n &\text{if}\;\o\in\O_n.
           \end{cases}\]
Given $\ep>0$ fix $n\in\bN$ so that
\begin{equation}\label{eq5.19}
\int_{\O_n}\negthickspace\negthickspace\|F\|dP<\ep
\end{equation}
By virtue of \eqref{eq5.19} the Aumann and the Hukuhara integrals of $F_n$ on $\O_n$ are contained in the ball $\ep B$, hence
\begin{gather*}
h\left( \Aint FdP,\;\text{(A)}\negthickspace\int_{\O^{\prime}\setminus\O_n}\negthickspace\negthickspace\negthickspace F_n dP\right)<\ep\\
h\left( \Hint FdP,\;\text{(H)}\negthickspace\int_{\O^{\prime}\setminus\O_n}\negthickspace\negthickspace\negthickspace F_n dP\right)<\ep.
\end{gather*}
Since $F_n$ is bounded on $\O^{\prime}\setminus\O_n$, the Aumann and the Hukuhara integrals of $F_n$ on $\O^{\prime}\setminus\O_n$ are equal. Consequently as $\ep>0$ is arbitrary, \eqref{eq5.13} is valid also when $F$ is not bounded. This completes the proof.
\end{proof}

\begin{proposition}\label{proposition5.6}
For each $F\in\bL^1(\O,\ffC)$
\begin{equation}\label{eq5.20}
\aint FdP=\hint\;(\cco F)dP.
\end{equation}
\end{proposition}
\begin{proof}
By virtue of Proposition \ref{proposition5.5},
\begin{equation}\label{eq5.21}
\aint FdP\subset\aint\;(\cco F)dP=\hint\;(\cco F)dP.
\end{equation}
To prove the reverse inclusion, let $\{G_n\}_{n=1}^{+\infty}\subset\bZ^1(\O,\ffC)$ be a sequence of countably-valued maps $G_n:\O\to\ffC$ converging to $F$ uniformly a.s. on $\O$. Suppose that each $G_n$ has a representation given by
\begin{equation*}
G_n(\o) = \sum_{i=1}^{+\infty}A_i^n \chi_{\O_i^n}(\o), \qquad \o\in\O,
\end{equation*}
where $\{A_i^n\}_{i=1}^{+\infty}\subset\ffC$, $\{\O_i^n\}_{i=1}^{+\infty}$ is a measurable partition of $\O$. Given $\ep>0$ take $n_0\in\bN$ so that
\begin{equation}\label{eq5.22}
h_{\infty}(G_n,F)<\ep, \qquad h_{\infty}(\cco G_n,\cco F)<\ep \qquad\text{for all}\; n\geq n_0,
\end{equation}
the second inequality being a consequence of the first one. We have
\begin{gather}\label{eq5.23}
\hint \;(\cco F)dP \subset \hint\;\overline{(\cco G_n+\ep B)}dP \qquad\text{by \eqref{eq5.22}}\nonumber\\
\subset \hint\cco G_n dP+2\ep B \\
=\aint G_n dP +2\ep B \qquad\qquad\text{by Proposition \ref{proposition5.4}}\nonumber\\
\subset \aint\;(\overline{F+\ep B})dP +2\ep B \qquad\qquad\text{by \eqref{eq5.22}}\nonumber\\
\subset \aint F\,dP +4\ep B, \nonumber
\end{gather}
where the latter inclusion holds by Proposition \ref{proposition5.2} (with $r=\ep$ and $\s=\ep/2$). Combining \eqref{eq5.21} and (\ref{eq5.23}.23), as $\ep>0$ is arbitrary, one obtains \eqref{eq5.20}. This completes the proof.
\end{proof}

\begin{definition}\label{def5.3}
The \textit{expectation} $E(F)$ of a random set $F\in\bL^1(\O,\ffC)$ is defined by
\begin{equation*}
E(F)=\aint FdP.
\end{equation*}
\end{definition}

As an immediate consequence of the above definition and Proposition \ref{proposition5.6} we have the following:
\begin{proposition}\label{proposition5.7}
For each random set $F\in\bL^1(\O,\ffC)$
\begin{equation*}
E(F)=\hint\; (\cco F)dP.
\end{equation*}
\end{proposition}

\section{Strong laws of large numbers}\label{section6}

In this section we present two versions of the strong law of large numbers for random sets with values in $\fffC$ or in $\ffC$, where $\bE$ is an arbitrary not necessarily separable Banach space. Since in our approach we use a R\text{\r{a}}dstr$\Ddot{\text o}$m embedding technique we start by reviewing some of its properties (see \cite{Radstrom1952, Aumann1965}).

In what follows the space $\fffC$ is equipped with the operations of \textit{addition} $X\mi Y$ and \textit{multiplication} $\l X$ by non negative scalars defined, for every $X,Y\in\fffC$ and $\l\geq 0$ as follows
\begin{equation*}
X\mi Y = \overline{\{x+y \;|\; x\in X, y\in Y\}}, \qquad \l X= \{\l x \;|\; x\in X\}.
\end{equation*}
Clearly $X\mi Y$ and $\l X$ are in $\fffC$.

The following Propositions \ref{proposition6.1}-\ref{proposition6.3} can be proved as in \cite{Hiai1984, Radstrom1952}.

\begin{proposition}\label{proposition6.1}
For arbitrary $X,Y,Z\in\fffC$ and $\l,\m\geq 0$ we have: $(a_1)$ $X\mi\{0\}=X$; $(a_2)$ $X\mi Y =Y\mi X$; $(a_3)$ $X\mi(Y\mi Z)=(X\mi Y)\mi Z$; $(a_4)$ $1\cdot X=X$; $(a_5)$ $\l(\m X)=(\l\m)X$; $(a_6)$ $\l (X\mi Y)=\l X\mi\l Y$; $(a_7)$ $(\l+\m)X=\l X\mi \m Y$.\\ Moreover the operations of addition and multiplication by non negative scalars are continuous in the topology generated by the Pompeiu-Hausdorff metric $h$ of $\fffC$.
\end{proposition}

\begin{proposition}\label{proposition6.2}
Let $A, C, Z$ be non empty subsets of $\bE$ and suppose that $C$ is convex and $Z$ is bounded. Then
\begin{equation*}
\overline{A+ Z}\subset \overline{C+ Z} \quad\Rightarrow\quad A\subset \overline{C}. 
\end{equation*}
\end{proposition}

\begin{proposition}\label{proposition6.3}
\begin{itemize}
\item[$(i)$] For $X,Y,Z \in\fffC$
\begin{equation*}
X\mi Z= Y\mi Z \quad\Longleftrightarrow\quad X=Y.
\end{equation*}
\item[$(ii)$] For $X,Y,Z\in \fffC$ and $\l\geq 0$
\begin{equation*}
h(X\mi Z, Y\mi Z)= h(X,Y) \qquad h(\l X, \l Y)=\l h(X,Y).
\end{equation*}
\end{itemize}
\end{proposition}

By virtue of Propositions \ref{proposition6.1}-\ref{proposition6.3}, adapting some arguments from \cite{Aumann1965, Hiai1984, Puri1985, Radstrom1952} one can establish the following

\begin{proposition}[R\text{\r{a}}dstr$\Ddot{\text o}$m embedding]\label{proposition6.4}
There exists a Banach space $(\bF,\|\cdot\|)$ and a map $J:\fffC\to\bV$, where $\bV=J(\fffC)\subset\bF$ such that:
\begin{itemize}
 \item[$(i)$] $J(\l X\mi\m Y)=\l J(X)+\m J(Y)$ for every $X,Y\in\fffC$ and $\l,\m\geq 0$;
\item[$(ii)$] $\|J(X)-J(Y)\|=h(X,Y)$ for every $X,Y\in\fffC$;
\item[$(iii)$] $\bV$ is a convex cone in $\bF$, complete under the metric induced by the norm of $\bF$. 
\end{itemize}
\end{proposition}

It is worth noting that, since $\fffC$ is a complete metric space under the Pompeiu-Hausdorff metric $h$, then any nonempty closed subset of $\fffC$ is a complete metric space under the induced metric.

A nonempty set $\cW\subset\fffC$ is called a \textit{semilinear complete metric space} if: $(a_1)$ $\cW$ is closed in $\fffC$; $(a_2)$ $\cW$ is stable under the operations of addition and multiplication by nonnegative scalars, $i.e.$, for every $X,Y\in\cW$ and $\l\geq 0$ we have $X\mi Y\in\cW$ and $\l X\in\cW$.

It is evident that the space $\fffC$ itself is a semilinear complete metric space.

Retaining the above notations we now prove a useful variant of Proposition \ref{proposition6.4}.

\begin{proposition}\label{proposition6.5}
Let $\cW\subset\fffC$ be a semilinear complete and separable metric space. Let $\tilde{J}$ be the restriction of $J$ to $\cW$. Then there exists a separable Banach space $(\tilde{\bF},\|\cdot\|)$ such that the map $\tilde{J}:\cW\to\bW$, where $\bW=\tilde{J}(\cW)\subset \tilde{\bF}$ has the following properties:\begin{itemize}
\item[$(i)$] $\tilde{J}(\l X\mi\m Y)=\l \tilde{J}(X)+\m \tilde{J}(Y)$ for every $X,Y\in\cW$ and $\l,\m\geq 0$;
\item[$(ii)$] $\|\tilde{J}(X)-\tilde{J}(Y)\|=h(X,Y)$ for every $X,Y\in\cW$;
\item[$(iii)$] $\bW$ is a convex cone in $\tilde{\bF}$, complete under the metric induced by the norm of $\tilde{\bF}$. 
\end{itemize}
\end{proposition}
\begin{proof}
Set $\tilde{\bF}=\overline{\text{span}\,\bW}$, where $\text{span}\,\bW$ denotes the linear span of $\bW$, and observe that $\tilde{\bF}\subset\bF$, since $\bW\subset \bF$ by Proposition \ref{proposition6.4}. Moreover $\tilde{\bF}$ is separable, because the separability of $\cW$ implies that $\bW$ as well as its linear span are separable. Therefore $\tilde{\bF}$, as a closed and separable linear subspace of $\bF$, is actually a separable Banach space uneder the induced norm of $\bF$. It remains to verify $(i)-(iii)$.\\
$(i)$ Let $X,Y\in\cW$ and $\l,\m\geq 0$. Since $\l X\mi\m Y\in\cW$, then by Proposition \ref{proposition6.4} one has
\begin{equation*}
\tilde{J}(\l X\mi\m Y)=J(\l X\mi\l Y)=\l J(X)+\m J(Y)=\l\tilde{J}(X)+\m\tilde{J}(Y).
\end{equation*}
$(ii)$ This is obvious, since for $X,Y\in\cW$,
\begin{equation*}
\|\tilde{J}(X)-\tilde{J}(Y)\|=\|J(X)-J(Y)\|=h(X,Y).
\end{equation*}
$(iii)$ The set $\bW\subset \tilde{\bF}$ is a convex cone. In fact let $s,t\geq 0$ and $\xi,\eta\in\cW$ be arbitrary. Then $\xi=J(X)$, $\eta=J(Y)$ for some $X,Y\in\cW$ and thus
\begin{equation*}
s\xi+t\eta=sJ(X)+tJ(Y)=J(sX\mi tY)=\tilde{J}(sX\mi tY),
\end{equation*}
since $sX\mi tY\in\cW$. Hence $s\xi+t\eta\in\cW$, and so $\bW$ is a convex cone. Moreover $\bW$ is closed in $\tilde{\bF}$. To see this, consider an arbitrary sequence $\{\xi_n\}_{n=1}^{+\infty}\subset \bW$ converging to some $\xi\in\tilde{\bF}$. For each $n\in\bN$ there exists $X_n\in\cW$ such that $\tilde{J}(X_n)=\xi_n$. By virtue of $(ii)$ the sequence $\{X_n\}_{n=1}^{+\infty}\subset\cW$ is Cauchy and thus, as $\cW$ is complete, it converges to some $X\in\cW$. Set $\tilde{\xi}=\tilde{J}(X)$. Since by $(ii)$
\begin{equation*}
\|\xi_n-\tilde{\xi}\|=\|\tilde{J}(X_n)-\tilde{J}(X)\|=h(X_n,X),
\end{equation*}
it follows that $\{\xi_n\}_{n=1}^{+\infty}$ converges to $\tilde{\xi}$ as $n\to+\infty$. By the uniqueness of the limit one has  $\xi=\tilde{\xi}$, $i.e.$, $\xi=\tilde{J}(X)$. Hence $\xi\in\bW$ and thus $\bW$ is closed in $\tilde{\bF}$. This completes the proof.
\end{proof}

\begin{proposition}\label{proposition6.6}
Let $\cW$ be a semilinear complete and separable metric space and, retaining the notation of Proposition \ref{proposition6.5}, let $\bW=\tilde{J}(\cW)$. Let $F\in\bL^1(\O,\cW)$. Then the map $\xi:\O\to\bW$ given by $\xi(\o)=\tilde{J}(F(\o))$, $\o\in\O$, is Bochner integrable on $\O$, $i.e.$, $\xi\in L^1(\O,\tilde{\bF})$. Moreover,
\begin{equation}\label{eq6.1}
\tilde{J}\left(\hint FdP\right) =\int_{\O}\negthickspace\negthickspace\xi dP,
\end{equation}
where the integral on the right-hand side is an element of $\bW$.
\end{proposition}
\begin{proof}
As $F\in\bL^1(\O,\cW)$, there is a sequence $\{G_n\}_{n=1}^{+\infty}\subset \bZ^1(\O,\cW)$ of countably-valued maps $G_n :\O\to\cW$ converging to $F$ uniformly a.s. in $\O$. Consider the corresponding sequence $\{\xi_n\}_{n=1}^{+\infty}$, where $\xi_n :\O\to\tilde{\bF}$ is given by $\xi_n(\o)=\tilde{J}(G_n(\o))$, $\o\in\O$. Clearly $\{\xi_n\}_{n=1}^{+\infty}\subset \bL^1(\O,\tilde{\bF})$, since each $\xi_n$ is countably-valued and, for some $k>0$, $\|\xi_n(\o)\|\leq \|F(\o)\|+k$, a.s. on $\o\in\O$. Moreover the sequence $\{\xi_n\}_{n=1}^{+\infty}\subset L^1(\O,\tilde{\bF})$ converges to $\xi$ uniformly a.s. in $\O$, because by Proposition \ref{proposition6.5}
\begin{equation*}
\|\xi_n(\o)-\xi(\o)\|=\|\tilde{J}(G_n(\o))-\tilde{J}(F(\o))\|=h(G_n(\o),F(\o)) \qquad\o\in\O,
\end{equation*}
and hence $\xi$ is measurable. Moreover $\|\xi(\o)\|\leq \|F(\o)\|+k$ a.s. on $\O$ and thus $\xi\in L^1(\O,\tilde{\bF})$. Furthermore, by Proposition \ref{proposition6.5}, for each $n\in\bN$ one has
\begin{equation*}
\tilde{J}\left(\hint G_n dP\right) =\int_{\O}\negthickspace\negthickspace\xi_n dP,
\end{equation*}
from which \eqref{eq6.1} follows at once by letting $n\to+\infty$. Moreover the integral on the right-hand side of \eqref{eq6.1} is an element of $\cW$, because each $\xi_n$ takes values in $\bW$, and $\bW$ is a convex and closed cone contained in $\tilde{\bF}$. This completes the proof.
\end{proof}

It is worth noting that, by Proposition \ref{proposition5.7}, the expectation $E(F)$ of any random set $F\in\bL^1(\O,\ffC)$ is a nonempty bounded closed and convex subset of $\bE$, $i.e.$, $E(F)=C$ where $C\in\fffC$.

\begin{theorem}\label{theo6.1}
Let $\cU$ be a nonempty separable subset of $\ffC$, equipped with the induced Pompeiu-Hausdorff metric $h$ of $\ffC$. Let $\{X_n\}_{n=1}^{+\infty}$ be a sequence of independent identically distributed (i.i.d.) random sets $X_n :\O\to\cU$ with expectation $E(X_n)=C$, $n\in\bN$, where $C\in\fffC$. Then
\begin{equation*}
\frac{\overline{\cco X_1(\o)+\dots+\cco X_n(\o)}}{n}\xrightarrow{h} C \qquad\text{as}\;n\to+\infty,\quad \text{a.s. on}\;\;\O.
\end{equation*}
\end{theorem}
\begin{proof}
Set $\cV=\{Z\in\fffC \;|\; Z=\cco X \;\text{where}\; X\in\cU\}$ and define
\begin{multline*}
\cW_0=\{ Z\in\fffC \;|\; Z=\l_1 X_1\mi\dots\mi \l_n X_n\qquad\qquad\qquad\qquad\qquad\\
 \text{for some}\; X_i\in\cV, \l_i\geq 0, i=1,\dots,n,\;n\in\bN\},
\end{multline*}
and
\begin{equation*}
\cW=\text{cl}_{\fffC}[\cW_0].
\end{equation*}
The set $\cV$ is separable. To see this let $\{X_n\}_{n=1}^{+\infty}\subset \cU$ be a sequence dense in $\cU$. Then the sequence $\{\cco X_n\}_{n=1}^{+\infty}\subset \cV$ is dense in $\cV$. In fact, given $\ep>0$ and $Z\in\cV$, $i.e.$, $Z=\cco X$ for some $X\in\cU$, then taking a $X_n$ so that $h(X_n,X)<\ep$ one has $h(\cco X_n,\cco X)<\ep$. Clearly, the separability of $\cV$ implies that both $\cW_0$ and $\cW$ are separable. As $\fffC$ is complete, and $\cW$ is closed in $\fffC$ it follows that $\cW$ equipped with the Pompeiu-Hausdorff metric $h$ is a complete metric space. In addition $\cW$ is stable under the operations of addition and multiplication by nonnegative scalars because $\cW_0$ is so. Hence, $\cW\subset\fffC$ is a semilinear complete and separable metric space. Then, by virtue of Proposition \ref{proposition6.5}, there exists a separable Banach space $\tilde{\bF}$ such that, denoting by $\tilde{J}$ the restriction of $J$ to $\cW$ and setting $\bW=\tilde{J}(\cW)$, the properties of Proposition \ref{proposition6.5} are satisfied.

Now consider the sequence of random sets $\cco X_n:\O\to\cW$. In view of \cite{Taylor1978}, $\{\cco X_n\}_{n=1}^{+\infty}$ is a sequence of i.i.d. random sets because $\{X_n\}_{n=1}^{+\infty}$ is so and, moreover, the map $\cco :\ffC\to\fffC$ is nonexpansive. For each $n\in\bN$ let $\xi_n:\O\to\bW$ be given by
\begin{equation*}
\xi_n(\o)=\tilde{J}(\cco X_n(\o)) \qquad \o\in\O.
\end{equation*}
As the map $\tilde{J}:\cW\to\bW$ is nonexpansive, it follows that $\{\xi_n\}_{n=1}^{+\infty}$ is a sequence of i.i.d. random variables and, moreover, each $\xi_n$ takes values in $\tilde{\bF}$, a separable Banach space. Furthermore for each $n\in\bN$,
\begin{equation*}
E(\xi_n)=\int_{\O}\negthickspace\negthickspace\xi_n dP= \int_{\O}\negthickspace\negthickspace\tilde{J}(\cco X_n)dP = \tilde{J}\left(\hint(\cco X_n)dP\right) =\tilde{J}(E(X_n))= \tilde{J}(C),
\end{equation*}
since the Hukuhara integral equals $E(X_n)$ by Proposition \ref{proposition5.7}, and $E(X_n)=C$, $n\in\bN$, by hypothesis. By virtue of the SLLN theorem for random variables in a separable Banach spaces (see \cite{Mourier1956}) it follows that
\begin{equation*}
\frac{\xi_1(\o)+\dots+\xi_n(\o)}{n}\to \tilde{J}(C) \qquad\text{as}\;n\to+\infty,\quad \text{a.s. on}\;\O.
\end{equation*}
Consequently, by virtue of Proposition \ref{proposition6.5},
\begin{equation*}
\frac{\overline{\cco X_1(\o)+\dots+\cco X_n(\o)}}{n}\xrightarrow{h} C \qquad\text{as}\;n\to+\infty,\quad \text{a.s. on}\;\O.
\end{equation*}
This completes the proof.
\end{proof}

\begin{theorem}\label{theo6.2}
Let $\cU$ be a nonempty separable subset of $\ffC$ equipped with the induced Pompeiu-Hausdorff metric $h$ of $\ffC$. Let $\{X_n\}_{n=1}^{+\infty}$ be a sequence of i.i.d. random sets $X_n:\O\to\cU$ with expectation $E(X_n)=C$, $n\in\bN$, where $C\in\fffC$. Furthermore suppose that
\begin{equation}\label{eq6.2}
\{X_n(\o)\}_{n=1}^{+\infty}\subset \cZ(\o) \qquad \text{a.s. on}\;\O,
\end{equation}
where, for $\o\in\O$ a.s., $\cZ(\o)$ is a nonempty compact subset of $\cU$. Then
\begin{equation*}
\frac{\overline{X_1(\o)+\dots+X_n(\o)}}{n}\xrightarrow{F} C \qquad\text{as}\;n\to+\infty,\quad \text{a.s. on}\;\O.
\end{equation*}
If, in addition, $C\in\fffC$ is compact, then
\begin{equation*}
\frac{\overline{X_1(\o)+\dots+X_n(\o)}}{n}\xrightarrow{h} C \qquad\text{as}\;n\to+\infty,\quad \text{a.s. on}\;\O.
\end{equation*}
\end{theorem}
\begin{proof}
By virtue of Theorem \ref{theo6.1}
\begin{equation*}
\frac{\overline{\cco X_1(\o)+\dots+\cco X_n(\o)}}{n}\xrightarrow{h} C \qquad\text{as}\;n\to+\infty,\quad \text{a.s. on}\;\O.
\end{equation*}
and thus, a fortiori,
\begin{equation*}
\frac{\overline{\cco X_1(\o)+\dots+\cco X_n(\o)}}{n}\xrightarrow{F} C \qquad\text{as}\;n\to+\infty,\quad \text{a.s. on}\;\O.
\end{equation*}
From the latter and \eqref{eq6.2}, by Theorem \ref{theo4.1}, it follows that
\begin{equation}\label{eq6.3}
\frac{\overline{X_1(\o)+\dots+X_n(\o)}}{n}\xrightarrow{F} C \qquad\text{as}\;n\to+\infty,\quad \text{a.s. on}\;\O.
\end{equation}
If $C\in\fffC$ is also compact, then from \eqref{eq6.3}, in view of Theorem \ref{theo4.2}, one has
\begin{equation*}
\frac{\overline{X_1(\o)+\dots+X_n(\o)}}{n}\xrightarrow{h} C \qquad\text{as}\;n\to+\infty,\quad \text{a.s. on}\;\O.
\end{equation*}
This completes the proof.
\end{proof}

\begin{corollary}
Let $\{X_n\}_{n=1}^{+\infty}$ be a sequence of i.i.d. random sets $X_n:\O\to\cZ$, where $\cZ$ is a compact subset of $\ffC$, with expectations $E(X_n)=C$, $n\in\bN$, where $C\in\fffC$. Then
\begin{equation*}
\frac{\overline{\cco X_1(\o)+\dots+\cco X_n(\o)}}{n}\xrightarrow{h} C \qquad\text{as}\;n\to+\infty,\quad \text{a.s. on}\;\;\O
\end{equation*}
and
\begin{equation*}
\frac{\overline{X_1(\o)+\dots+X_n(\o)}}{n}\xrightarrow{F} C \qquad\text{as}\;n\to+\infty,\quad \text{a.s. on}\;\;\O.
\end{equation*}
\end{corollary}
\begin{proof}
The statement follows from Theorems \ref{theo6.1} and \ref{theo6.2}, since $\cZ$ is a separable subset of $\ffC$.
\end{proof}

\begin{Remark}
It is worth noting that in Theorems \ref{theo6.1} and \ref{theo6.2} the separability assumption on $\cU$ and the compactness assumption on $\cZ(\o)$ do not imply that the corresponding sets $U\subset\bE$ and $V\subset\bE$ given by
\begin{equation*}
U=\bigcup\{X\;|\;X\in\cU\},\qquad\qquad V=\overline{\bigcup\{X_n(\o)\;|\;n\in\bN\}}
\end{equation*}
be respectively separable, compact. To see this it suffices to consider the trivial examples $\cU=\{\bar{B}\}$ and $X_n(\o)=\bar{B}$, for $\o\in\O$, $n\in\bN$, where $\bar{B}$ denotes the closed unit ball in an infinite dimensional non-separable Banach space.
\end{Remark}

\bibliographystyle{amsplain}
%\providecommand{\bysame}{\leavevmode\hbox to3em{\hrulefill}\thinspace}
%\providecommand{\href}[2]{#2}

%\bibliography{biblio}

\end{document}